\documentclass[12pt,reqno]{amsart}
\usepackage{amsmath}
\usepackage{amsfonts}
\usepackage{amssymb}
\usepackage{verbatim}
\usepackage{epsfig}  		% For postscript
\usepackage{epic,eepic}       % For epic and eepic output from xfig
\usepackage{tikz}
\usepackage{tikz-cd}
\usetikzlibrary{matrix, calc, arrows}
\usepackage{hyperref}

\usepackage {amsmath,amssymb,verbatim,geometry}
\newtheorem{thm}{Theorem}[section]
\newtheorem{cor}[thm]{Corollary}
\newtheorem{lem}[thm]{Lemma}
\newtheorem{mydef}[thm]{Definition}
\newtheorem{rem}[thm]{Remark}
\newtheorem{ex}[thm]{Example}
\newtheorem{question}[thm]{Question}

\newtheorem{prop}[thm]{Proposition}

\newcommand{\xRightarrow}[2][]{\ext@arrow 0359\Rightarrowfill@{#1}{#2}}

\newtheorem{thmmain}{Main Theorem}

\begin{document}

\title{A partial comparison of stability notions in K\"ahler geometry}
\author{Zakarias Sj\"ostr\"om Dyrefelt}
\email{zsjostro@ictp.it}
\address{The Abdus Salam International Centre for Theoretical Physics (ICTP), Str. Costiera, 11, 34151 Trieste TS, Italy.}

\begin{abstract}
In this follow up work to \cite{SD1, DervanRoss, Dervanrelative, SD2} we introduce and study a notion of geodesic stability restricted to rays with prescribed singularity types. A number of notions of interest fit into this framework, in particular algebraic- and transcendental K-polystability, equivariant K-polystability, and the geodesic K-polystability notion introduced by the author in \cite{SD2}. We provide a partial comparison of the above notions, and show equivalence of some of these notions provided that the underlying manifold satisfies a certain \emph{weak cscK} condition. As an application this proves K-polystability of a new family of cscK manifolds with irrational polarization. 
\end{abstract}

\maketitle
%\tableofcontents

\section{Introduction}

\noindent 
A fundamental open problem in K\"ahler geometry is the Yau-Tian-Donaldson (YTD) conjecture, which predicts that existence of canonical metrics (in the sense of Calabi \cite{Calabiextremal}) should be equivalent to a suitable stability notion in algebraic geometry. In the case of Fano manifolds $(X,-K_X)$ equipped with the anticanonical polarization, the conjecture has been proven with respect to the classical algebraic notion of K-stability with roots in geometric invariant theory \cite{CDSone, CDStwo, CDSthree}. 
For polarized manifolds $(X,L)$, or even completely arbitrary K\"ahler manifolds $(X,\omega)$, finding the precise stability notion that makes the conjecture hold is then a central part of the problem. Indeed, it is debated which ones of a rapidly growing number of proposed stability notions (transcendental/algebraic/equivariant/filtration/uniform K-stability) should be the most relevant to the conjecture, and the relationship between these a priori differing notions is largely unexplored. In this follow up work to \cite{SD1, DervanRoss, Dervanrelative, SD2} we aim to address this knowledge gap, by means of comparing some of the central stability notions in K\"ahler geometry to each other.

We will focus on the case of constant scalar curvature K\"ahler (cscK) metrics and related algebro-geometric stability notions. In particular we will investigate possible natural comparisons with the geodesic stability notion used in the recent proof of the properness conjecture, due to Chen-Cheng \cite{ChenChengII, ChenChengIII}. 
In this context, a notion of special interest to us is the notion of \emph{geodesic K-polystability}, which was introduced in \cite{SD2}. 
This is a new notion that means that $(X,[\omega])$ is K-semistable, and moreover, the Donaldson-Futaki invariant vanishes precisely for the test configurations whose ``associated geodesic ray'' is induced by a holomorphic vector field  (in a sense made precise in the aforementioned paper \cite{SD2}). As such, it can be interpreted as a weaker version of the geodesic stability notion used by Chen-Cheng \cite{ChenChengII}, which is is in turn known to be equivalent to existence of cscK metrics \cite{ChenChengI, ChenChengII, ChenChengIII}. This is particularly interesting in order to better understand the relationship between geodesic stability and the algebraic notions of K-polystability. Ultimately, such a comparison is precisely what is required to prove or disprove the YTD conjecture. 

In view of the classical correspondence between geodesic rays and test configurations, see e.g. \cite{ArezzoTian, PRS, Berman, SD1, SD2, DervanRoss, ChenTang, WR, BDL, BHJ1} and references therein, there are a number of reasons to believe that geodesic K-polystability is a natural stability notion. First of all, it was proven in \cite{BDL, SD2} that constant scalar curvature K\"ahler (cscK) manifolds are geodesically K-polystable, thus proving one direction of a natural YTD conjecture in this setting. Moreover, if $\mathrm{Aut}_0(X) = \emptyset$ and the underlying class is cscK, then geodesic K-polystability is equivalent to the usual K-polystability notion \cite{SD2}. It was also checked by R. Dervan in an appendix to \cite{SD2} that geodesic K-polystability implies equivariant K-polystability (as introduced in his paper \cite{Dervanrelative}), generalizing a notion introduced in \cite{Donaldsontoric, GaborStoppa}, which is conjectured to be equivalent to the cscK condition. 
When $\mathrm{Aut}_0(X) \neq \emptyset$ or the underlying K\"ahler class is not cscK, the relationship to the full non equivariant K-polystability notion however remains an open problem (of importance to understanding the YTD conjecture).

To study the above stability notions we introduce the terminology of \emph{stability loci} in the K\"ahler cone: Denote by \emph{K-polystable locus} the set of K\"ahler classes $\alpha$ in the K\"ahler cone of $X$ such that $(X,\alpha)$ is K-polystable, and use similar terminology for other stability notions. Likewise, we say that the \emph{cscK locus} is the set of all K\"ahler classes $\alpha$ on $X$ which contain a cscK metric. In particular, the YTD conjecture then translates to the statement that the cscK locus coincides with the K-polystable locus. This way stability may be considered not as a question on a single given polarization, but as a question about characterizing a certain subset of the K\"ahler cone. This is sometimes a useful point of view, as we shall see in this note. We may in particular ask the following broad but central questions:
How can we compare the cscK locus and the various stability loci? What stability loci coincide in the K\"ahler cone of $X$ (i.e. which stability notions are equivalent)? 
In this note we will give some partial answers to the second part of this question, and set up the framework for continuing to study such problems in future work.

\subsection{Comparing geodesic stability and K-stability}

The main results of this paper are partial results towards comparing geodesic stability in the sense of Chen-Cheng \cite{ChenChengIII} with classical K-stability notions, as well as (transcendental) K-polystability of $(X,[\omega])$ in the sense of \cite{SD2}. 
%Motivated by its natural interpretation as a special case of geodesic stability in the sense of Chen-Cheng \cite{ChenChengIII} (see Section \ref{Section Sgeodstability}), we focus on comparing geodesic K-polystability with other more classical K-stability notions, as well as (transcendental) K-polystability of $(X,[\omega])$ in the sense of \cite{SD2}. 
%%In particular, it is natural to compare geodesic K-polystability and (transcendental) K-polystability of $(X,[\omega])$ in the sense of \cite{SD2}. 
The status of the comparison problem for stability notions in K\"ahler geometry is as follows: For arbitrary compact K\"ahler manifolds $(X,\omega)$ (such that the associated K\"ahler class $[\omega] \in H^{1,1}(X,\mathbb{R})$ is possibly irrational) we have inclusions $$\text{cscK locus} \subseteq \text{Geodesically K-polystable locus},$$ and it is straightforward to see that K-polystable locus $\subseteq$ Geodesically K-polystable locus. 
It is however open whether the cscK locus is included in the K-polystable locus, and it is unknown what is the precise relationship between the geodesically K-polystable locus and the K-polystable locus (especially if the underlying class does not admit a cscK metric). These are questions that concern the relationship between test configurations and their associated geodesic rays in the space of K\"ahler metrics. Indeed, the problem here posed is equivalent to asking if a test configuration is a product (in the sense that $\mathcal{X}_{\pi^{-1}(\mathbb{C})} \simeq X \times \mathbb{C}$) precisely if its associated geodesic ray is induced by a holomorphic vector field on $X$ (Definition \ref{Definition induced vector field geodesic}). What is immediately clear is that if there is a geodesic ray induced by a holomorphic vector field coming from a non-product configuration, then the underlying polarization $(X,[\omega])$ is not K-polystable, but it is a priori unclear what happens for other classes in the K\"ahler cone of $X$. 
As a first main result, we prove the following: 
%make the following observation: 

\begin{thmmain} \label{Theorem dichotomy main}
%\label{Theorem dichotomy}
Let $(X,\omega)$ be a compact K\"ahler manifold and suppose that the K-polystable locus $\neq \emptyset$. Then 
$(X,[\omega])$ is K-polystable if and only if it is geodesically K-polystable.
%the K-polystable locus equals the geodesically K-polystable locus. 
%\begin{enumerate}
 %   \item The K-polystable locus is empty in the K\"ahler cone of $X$.
 %   \item The K-polystable locus coincides with the geodesically K-polystable locus.  
%\end{enumerate}
\end{thmmain}

\noindent In particular, this gives a partial answer to the question of comparing K-polystability and geodesic K-polystability. %For manifolds with non-empty K-polystable locus the following result also holds:
In light of \cite[Theorem 1.1]{SD2} we also have the following first result of K-polystability for cscK manifolds that are not necessarily polarized and are allowed to admit holomorphic vector fields:

%the notions are equivalent provided that a K-polystable K\"ahler class exists. 

\begin{cor} \label{Corollary main}
Let $(X,\omega)$ be a cscK manifold with K-polystable locus $\neq \emptyset$. 
%If $[\omega]$ admits a cscK metric, 
Then $(X,[\omega])$ is K-polystable. 
\end{cor}

\noindent Note that his proves one direction of the YTD conjecture for a new family of compact K\"ahler manifolds $(X,\omega)$ with irrational polarization, i.e. when $[\omega] \in H^{1,1}(X,\mathbb{R})$ is an arbitrary K\"ahler class on $X$ not necessarily in the rational lattice $H^2(X,\mathbb{Q})$. 
As part of the proof we in particular obtain the following result, which sheds additional light on the connection between geodesic rays and test configurations, extending results of \cite{BDL}. 

\begin{thm} \label{Theorem dichotomy}
Let $(X,\omega)$ be a compact K\"ahler manifold and suppose that the K-polystable locus $\neq \emptyset$. Suppose that $(\mathcal{X}, \mathcal{A})$ is a test configuration for $(X,[\omega])$. Then the following are equivalent:
\begin{itemize}
\item$\mathcal{X}_{\pi^{-1}(\mathbb{C})} \simeq X \times \mathbb{C}$ 
%\item $\mathcal{X}_0 \simeq X$ \textbf{(???)}
\item The associated geodesic ray is induced by a holomorphic vector field on $X$ \emph{(Definition \ref{Definition induced vector field geodesic})}. 
\end{itemize}
 %   \item The K-polystable locus is empty in the K\"ahler cone of $X$.
 %   \item The K-polystable locus coincides with the geodesically K-polystable locus.  
%\end{enumerate}
\end{thm}

\noindent Main theorem \ref{Theorem dichotomy main} and Theorem \ref{Theorem dichotomy} together strengthen the expectation that the notions of K-polystability and geodesic K-polystability are equivalent in general.
Combined with \cite[Theorem 1.1]{SD2}, Theorem \ref{Theorem dichotomy} moreover reduces the statement ``cscK manifolds are K-polystable'', which is an important problem still open for arbitrary K\"ahler manifolds, to understanding whether the K-polystable locus is non-empty. 
%\end{rem} 

As a natural family of examples we may consider compact K\"ahler manifolds that we shall refer to as \emph{weakly cscK}, i.e. such that the cscK locus $\neq \emptyset$ in the K\"ahler cone $\mathcal{C}_X$ of $X$. Indeed, polarized weakly cscK manifolds $(X,L)$ satisfy the hypothesis that the K-polystable locus $\neq \emptyset$, since cscK then implies K-polystability (see \cite{Berman, BDL}).  
Moreover, there are many interesting concrete examples of weakly cscK manifolds; in particular any K\"ahler-Einstein Fano manifold is weakly cscK. 
We have the following immediate corollary of Theorem \ref{Theorem dichotomy main} for weakly cscK polarized manifolds:

\begin{thm}
Let $(X,L)$ be a polarized weakly cscK manifold. Then 
\begin{enumerate}
    \item $(X,L)$ is  K-polystable if and only if it is  geodesically K-polystable.
    \item $(X,L)$ is equivariantly geodesically K-polystable if and only if it is equivariantly K-polystable.
\end{enumerate}
\end{thm}

\noindent The stability notions referred to in $(1)$ are the classical (algebraic) K-polystability and (algebraic) geodesic K-polystability for polarized manifolds, see Section \ref{Section definitions Kps} for precise definitions. In $(2)$ we say that $(X,[\omega])$ is \emph{equivariantly geodesically K-polystable} if and only if it is geodesically K-polystable with respect to equivariant test configuratons (see \cite{Dervanrelative, SD2} for the K\"ahler case). Hence, this extends result of \cite{BDL} from the case of polarized cscK manifolds to weakly cscK polarized manifolds. 
Moreover, the result $(2)$ holds also for arbitrary compact K\"ahler manifolds, using the formalism for \cite{SD1, DervanRoss}. The result $(1)$ can be checked for arbitrary compact K\"ahler manifolds provided that the automorphism group is discrete, but remains an open question in general.

We also record the following comparison of stability notions, which holds even for non-polarized K\"ahler manifolds $(X,\omega)$ (for the compatibility notion see Section \ref{Section preliminary rays} and references therein): 
\begin{rem}
Suppose that $(X,\omega)$ is a weakly cscK K\"ahler manifold with $\mathrm{Aut}_0(X)$ discrete. Then $(X,[\omega])$ is uniformly K-stable if and only if $(X,[\omega])$ is coercive with respect to the set of subgeodesic rays compatible with a relatively K\"ahler test configuration for $(X,[\omega])$. Likewise, $(X,[\omega])$ is K-stable if and only if $(X,[\omega])$ is geodesically stable with respect to the set of subgeodesic rays compatible with a relatively K\"ahler test configuration for $(X,[\omega])$.
\end{rem}

%\medskip

%\subsection{Equivalence of notions of product configuration}
\subsection{A remark on triviality of test configurations}

Another corollary of the techniques of this paper concerns the equivalence of various notions of product configurations occurring in the literature. This is interesting in its own right, since it addresses the question of equivalence of several commonly seen (and a priori different) candidate notions of K-polystability. Indeed, these notions have in commmon that they ask that the so called Donaldson-Futaki invariant $\mathrm{DF}(\mathcal{X},\mathcal{A})$ is non-negative for all test configurations $(\mathcal{X},\mathcal{A})$ for $(X,\alpha)$, with equality if and only if $(\mathcal{X},\mathcal{A})$ is a ``product'', in a suitable sense. Addressing a question asked in the author's thesis \cite{Zakthesis}, the following result proves that several commonly seen notions of product configuration are in fact equivalent: 

\begin{thm} 
%A polarized \textbf{weakly cscK} manifold $(X,L)$ is type A K-polystable if and only if it is type B K-polystable. 
Suppose that $(X,L)$ is a
polarized weakly cscK manifold. Let $(\mathcal{X}, \mathcal{L})$ be a relatively K\"ahler test configuration for $(X,L)$, with associated geodesic ray $(\varphi_t)_{t \geq 0}$. Then the following are equivalent:
\begin{enumerate}
    \item $\mathcal{X}_{\pi^{-1}(\mathbb{C})} \simeq X \times \mathbb{C}$
    \item $\mathcal{X}_{\pi^{-1}(\Delta_r)} \simeq X \times \Delta_r$ for each $r > 0$, where $\Delta_r := \{ z \in \mathbb{C} \; \vert \; |z| < r \}$. 
    \item $\mathcal{X}_0 \simeq X$
    \item The associated geodesic ray $(\varphi_t)_{t \geq 0}$ is induced by a holomorphic vector field $V$ on $X$.
\end{enumerate}
\end{thm}

\noindent As before, the point is that this holds even if the underlying K\"ahler class $c_1(L)$ does not admit any cscK metrics, as long as $(X,L)$ is weakly cscK. In fact, note that if we want the K-polystable locus to contain the properness locus, then there is no choice but to define products as objects whose associated subgeodesic rays satisfy $$\inf_{g \in G} \mathrm{J}(g.\varphi_t) = 0,$$ where $G := \mathrm{Aut}_0(X)$ is the connected component of the automorphism group of $X$, the action $g.\varphi$ on potentials is defined as in Section \ref{Section geodesics and action}, and
$$
\mathrm{J}(\varphi) := \int_X \varphi \omega^n - \frac{1}{n+1} \sum_{k = 0}^{n} \int_X \varphi \omega^k \wedge \omega_{\varphi}^{n-k}.
$$
\noindent This is why geodesic K-polystability is such a natural notion, because it is ``the most obvious'' notion satisfying the above requirement. 

The techniques used to prove the above results rely on understanding paths of test configurations when changing the underlying class, and in particular the existence of such paths that preserve the associated geodesic ray. 

%\medskip

\subsection{Idea of the proofs: Special paths of test configurations and an injectivity lemma}

Let $\alpha, \beta \in H^{1,1}(X,\mathbb{R})$ be two K\"ahler classes on $X$. In order to study how K-polystability notions vary as we vary the underlying class, one has to understand how to relate test configurations for $(X,\alpha)$ to test configurations of $(X,\beta)$. 
A first straightforward observation is the following (part $(1)$ on convex combinations of test configurations should be compared to e.g. \cite{Antoniso} in the setting of polarized manifolds, part $(2)$ is a direct consequence of the intersection theoretic point of view due to \cite{Wang, Odaka, Bermantransc, SD1, DervanRoss} and part $(3)$ is a direct consequence of \cite{SD2}): 

\begin{thm} \label{Theorem convex combination}
Let $\alpha, \beta \in \mathcal{C}_X$ and set $\alpha_s := (1-s)\alpha + s\beta$, for $s \in [0,1]$. Suppose that $(\mathcal{X},\mathcal{A})$ and $(\mathcal{X},\mathcal{B})$ is a relatively K\"ahler test configuration for $(X,\alpha)$ and $(X,\beta)$ respectively. Then
\begin{enumerate}
    \item $(\mathcal{X}, (1-s)\mathcal{A} + s \mathcal{B})$ is a relatively K\"ahler test configuration for $(X,\alpha_s)$.
     \item The maps $[0,1] \ni s \mapsto \mathrm{DF}(\mathcal{X}, (1-s)\mathcal{A} + s \mathcal{B})$ and $[0,1] \ni s \mapsto \mathrm{J}^{\mathrm{NA}}(\mathcal{X}, (1-s)\mathcal{A} + s \mathcal{B})$ are continuous.
    \item Suppose that $\rho_A(t)$ and $\rho_B(t)$ are the uniquely associated geodesic rays respectively, and write $\rho_s(t) := (1-s)\rho_A(t) + s \rho_B(t)$ If $\alpha_s = [\omega_s]$, then 
    $$
    \mathrm{DF}(\mathcal{X}, (1-s)\mathcal{A} + s \mathcal{B}) = \lim_{t \rightarrow +\infty} t^{-1} \mathrm{M}_{\omega_t}(\rho_s(t)) - ((\mathcal{X}_{0,red} - \mathcal{X}_0) \cdot \mathcal{A}^n). 
    $$
\end{enumerate}
\end{thm}

\noindent In practice, it is however not a given to know something about the set of test configurations for a different polarization than the one considered. A key question then becomes: \emph{How can one relate the test configurations for $(X,\alpha)$ to the test configurations for $(X,\beta)$?} 
In this direction, we prove the following extended version of the injectivity lemma \cite[Theorem 1.8]{SD2} (now allowing for a change of the underlying class): 

\begin{thmmain}\label{Theorem extended injectivity lemma}
Let $\alpha := [\omega]$ and $\beta := [\theta]$ be K\"ahler classes on $X$. Suppose that there is a subgeodesic ray $\rho(t) \in \mathrm{PSH}(X,\omega) \cap \mathrm{PSH}(X,\theta)$ which is compatible with two relatively K\"ahler test configurations $(\mathcal{X},\mathcal{A})$ and $(\mathcal{Y}, \mathcal{B})$ for $(X,\alpha)$ and $(X,\beta)$ respectively. Then 
%$\mathcal{X} = \mathcal{Y}$. More precisely, 
the canonical $\mathbb{C}^*$-equivariant isomorphism $\mathcal{X} \setminus \mathcal{X}_0 \rightarrow \mathcal{Y} \setminus \mathcal{Y}_0$  extends to an isomorphism $\mathcal{X} \rightarrow \mathcal{Y}$. 
\end{thmmain}

\noindent A slightly more precise result relating test configurations of $(X,\alpha)$ to those of $(X,\beta)$ is given below: 

\begin{thmmain} \label{Theorem second main}
Let $\alpha \in \mathcal{C}_X$ and suppose that $(\mathcal{X},\mathcal{A})$ is a relatively K\"ahler smooth and dominating test configuration for $(X,\alpha)$. Then, for each $\beta \in \mathcal{C}_X$ there is a $\lambda > 0$ such that $\lambda \beta > \alpha$ and a relatively K\"ahler test configuration $(\mathcal{Y}, \mathcal{B})$ for $(X,\lambda \beta)$ such that 
\begin{enumerate}
    \item $\mathcal{Y} = \mathcal{X}$,
    \item The test configurations $(\mathcal{X},\mathcal{A}) \sim (\mathcal{Y}, \mathcal{B})$, i.e. there is a geodesic ray $\rho(t)$ compatible with both. 
\end{enumerate}
In particular, if $\alpha = [\omega]$ and $\lambda \beta = [\theta]$, then we have 
    $$
    \mathrm{DF}(\mathcal{X}, \mathcal{A}) = \lim_{t \rightarrow +\infty} t^{-1} \mathrm{M}_{\omega}(\rho(t)) - ((\mathcal{X}_{0,red} - \mathcal{X}_0) \cdot \mathcal{A}^n)
    $$
    and
    $$
     \mathrm{DF}(\mathcal{X}, \mathcal{B}) = \lim_{t \rightarrow +\infty} t^{-1} \mathrm{M}_{\theta}(\rho(t)) - ((\mathcal{X}_{0,red} - \mathcal{X}_0) \cdot \mathcal{B}^n).
    $$
\end{thmmain}

\noindent This result has a number of straightforward applications, in particular to proving Theorem \ref{Theorem dichotomy}. 

%\medskip

%\subsection{Applications to the topology of the K-semistable and uniformly K-stable loci}
\subsection{Topology of the K-semistable and uniformly K-stable loci}

As an application the techniques of this paper also yield some basic properties of the K-semistable and uniformly K-stable loci, where uniform K-stability is defined with respect to the norm $\mathrm{J}^{\mathrm{NA}}$, i.e. $(X,\alpha)$ is uniformly K-stable if there is a $\delta > 0$ such that $\mathrm{DF}(\mathcal{X},\mathcal{A}) \geq \delta \mathrm{J}^{\mathrm{NA}}(\mathcal{X},\mathcal{A})$ for all relatively K\"ahler test configurations $(\mathcal{X}, \mathcal{A})$ for $(X,\alpha)$ (see Section \ref{Section topology} for further details). In particular, the techniques of variation of the underlying class in the K\"ahler cone (Theorem \ref{Theorem convex combination}) immediately yield the following characterization of the K-semistable locus of the K\"ahler cone (cf. \cite[Theorem G]{Antoniso} for an analogous result in the projective setting). 

\begin{thm} \label{Theorem kss}
The K-semistable locus is closed in Euclidean topology in the K\"ahler cone of $X$.
\end{thm}

\noindent Recall moreover that the cscK locus is open relative to the Futaki vanishing locus (see \cite{LeBrunSimanca}), i.e. the cscK locus can be written as $U \cap \mathcal{C}_F$, where $U$ is an open set in the K\"ahler cone $\mathcal{C}_X$). As a consequence of this, the K-semistable and cscK loci can only coincide whenever they both equal $\emptyset$ or $\mathcal{C}_X$. The fact that these stability loci are in general not equal has been known by means of counterexamples (see e.g. \cite{Tian} and \cite[Corollary 1.2]{Keller}), but this yields a complementary perspective on this question.
Since the K-polystable locus is moreover expected to be open relative to the Futaki vanishing locus, we expect in the same way that the set of strictly K-semistable K\"ahler classes form the complement of an open set inside a closed set in the K\"ahler cone. 
As before, it is known by example (see \cite{Tian, Keller}) that strictly K-semistable classes exist, but this would give some additional information (for example we would expect strictly K-semistable to exist in abundance, except exceptional cases.) 

In case the automorphism group is discrete, we may also give a similar characterization of the uniformly K-stable locus. To do this we associate to each K\"ahler class $\alpha \in \mathcal{C}_X$ its 'stability threshold', that is
$$
\Delta(\alpha) := \sup \{ \delta > 0 \; \vert \mathrm{DF}(\mathcal{X},\mathcal{A}) \geq \delta \mathrm{J}^{\mathrm{NA}}(\mathcal{X},\mathcal{A})\} > -\infty,
$$
where the supremum is taken over all relatively K\"ahler test configurations $(\mathcal{X},\mathcal{A})$ for $(X,\alpha)$. Moreover, introduce the level sets
$$
\mathcal{U}_{\delta} := \{ \alpha \in \mathcal{C}_X \; \vert \; \Delta(\alpha) \geq \delta \}.
$$
\noindent We then make the following observation: 

\begin{thm} \label{Theorem uks}
The uniformly K-stable locus can be written as a union
$$
\mathcal{U} := \bigcup_{\delta > 0} \mathcal{U}_{\delta},
$$
where each set $\mathcal{U}_{\delta}$ is closed in the Euclidean topology in the K\"ahler cone. 
\end{thm}

\noindent Note that the K-semistable locus equals $\mathcal{U}_0$, so Theorem \ref{Theorem kss} is a special case of Theorem \ref{Theorem uks}. 

\subsection{Organization of the paper}
The goal of Sections \ref{Section 2} - \ref{Section proof of main results} is to rigorously clarify how to view various K-stability notions as special cases of the classical geodesic stability notion (which is now known to be equivalent to existence of cscK metrics, due to recent progress of Chen-Cheng \cite{ChenChengII, ChenChengIII}). In order to do this, some standard preliminary notions are recalled in Section \ref{Section 2}. The slightly non-standard geodesic stability notion used in this paper is discussed in Section \ref{Section geodesic stability definition}. In Section \ref{Section 3} the definitions of a wide variety of (transcendental) K-polystability notions are given. In Section \ref{Section proof of main results} we compare these notions, and give proofs of other main results. Section \ref{Section 5} contains three applications of the methods used in our proof, in particular to weakly cscK manifolds and basic topological properties of the K-semistable and uniformly K-stable loci. 
The formalism for test configurations used in this paper is based on the notions introduced in \cite{SD1, DervanRoss, SD2, Dervanrelative}.

%\subsection{Acknowledgements}
%It is my pleasure to thank the referee for helpful remarks and comments. 

\bigskip

\section{Geodesic stability notions in K\"ahler geometry} \label{Section 2}

\noindent Throughout this paper, let $(X,\omega)$ be a compact K\"ahler manifold. Let $n := \mathrm{dim}_{\mathbb{C}}(X)$ be the complex dimension of $X$. Write $$V := \int_X \omega^n := ( \alpha^n)_X$$ for the K\"ahler volume of $X$. 

\subsection{Background on the space of K\"ahler metrics and geodesics} \label{Section geodesics and action}

Consider the space 
$$
\mathcal{H}_{\omega} := \{ \varphi \in C^{\infty}(X) \; \vert \; \omega_{\varphi} := \omega + dd^c\varphi > 0 \}, \; \; (dd^c := \frac{i}{2\pi}\partial \bar{\partial})
$$
of smooth K\"ahler potentials. In a landmark paper by Mabuchi \cite{Mabuchisymplectic} it was shown that $\mathcal{H}_{\omega}$ is a Riemannian symmetric space (of infinite dimension, with $T_{\varphi}\mathcal{H} \simeq C^{\infty}(X)$). 
Following Darvas \cite{Darvas14, Darvas15, Darvas, DR} and others we however privilege the point of view of considering $\mathcal{H}$ as a path metric space endowed with a certain \emph{Finsler metric} $d_1$. To introduce it, let $d_1: \mathcal{H}_{\omega} \times \mathcal{H}_{\omega} \rightarrow \mathbb{R}_+$ be the path length pseudometric associated to the weak Finsler metric on $\mathcal{H}_{\omega}$ defined by
$$
||\xi||_{\varphi} := V^{-1} \int_X |\xi|\omega_{\varphi}^n, \; \; \; \xi \in T_{\varphi}\mathcal{H}_{\omega} = \mathcal{C}^{\infty}(X). 
$$
More explicitly, if $[0,1] \ni t \mapsto \phi_t$ is a smooth path in $X$, then let
$$
l_1(\phi_t):= \int_0^1 ||\dot{\phi}_t||_{\phi_t} dt
$$
be its length, and set
$$
d_1(\varphi, \psi) = \inf \left\{ l_1(\phi_t), \; \; (\phi_t)_{0\leq t \leq 1} \subset \mathcal{H}_{\omega}, \; \phi_0 = \varphi, \; \phi_1 = \psi \right\},
$$
where the infimum is taken over smooth paths $t \mapsto \phi_t$ as above. It can then be seen that $(\mathcal{H}_{\omega}, d_1)$ is a metric space which is \emph{not} complete (see \cite[Theorem 2]{Darvas} and the survey article \cite{Darvassurvey} for details and background). 
The completion  $\mathcal{E}$ of $(\mathcal{H}_{\omega},d_1)$ was described by Darvas \cite{Darvas15}. For the purpose of discussing energy functionals and geodesic stability we will in particular consider the subspace $\mathcal{E}^1 \subset \mathcal{E}$ of $\omega$-psh functions of finite $L^1$-energy, i.e. the subspace of all $\varphi \in \mathcal{E}$ such that
$$
\int_X |\varphi| \omega_{\varphi}^n < +\infty.
$$

\subsubsection{Group actions}

Let $G := \mathrm{Aut}_0(X)$ be the connected component of the complex Lie group of biholomorphisms of $(X,J)$, whose Lie algebra consists of real vector fields $V$ satisfying $\mathcal{L}_V J = 0$. For each $g \in G$ we then have $[g^*\omega] = [\omega]$ (this follows from Moser's trick in symplectic geometry, see e.g. \cite[Chapter III.7]{lecturessymplectic}). The group $G$ thus acts naturally on the space $\mathcal{K} := \{\omega_{\varphi}:= \omega + dd^c\varphi: \varphi \in \mathcal{C}^{\infty}(X), \omega_{\varphi} > 0 \}$ of K\"ahler metrics on $X$, so that
$
g \cdot \xi := g^*\xi, \; g \in G, \; \xi \in \mathcal{K}.
$
The space $\mathcal{K}$ is moreover in one-to-one correspondence with the space $\mathcal{H}_0 := \mathcal{H} \cap E^{-1}(0)$ of normalized K\"ahler potentials. Following \cite[Section 5.2]{DR} the group $G$ therefore also acts on $\mathcal{H}_0$, so that $g \cdot \varphi$ is the unique element in $\mathcal{H}_0$ satisfying $g \cdot \omega_{\varphi} = \omega_{g \cdot \varphi}$.
As in \cite[Lemma 5.8]{DR}) one may moreover show that
\begin{equation} \label{equation action}
g \cdot \varphi =  g \cdot 0 + \varphi \circ g. 
\end{equation}
By the $dd^c$-lemma the function $g \cdot 0$ is smooth, hence bounded, on $X$.  

\subsubsection{Geodesics in the space of K\"ahler metrics}

There is also a natural notion of geodesic (and subgeodesic) rays in $\mathcal{H}$. To define it, suppose that $I \subseteq (0, +\infty)$ is an open interval. Let $I \ni t \mapsto \varphi_t$ be any curve of functions on $X$. Then $(\varphi_t)_{t \in I}$ can be identified with an $S^1$-invariant function $\Phi$ on $X \times \Delta_I$, where $\varphi_t(x) = \Phi(x,e^{-t+is})$, and 
$$
\Delta_I := \{\tau \in \mathbb{C} \; \vert \;  -\log|\tau| \in I \}.
$$
We will be mainly interested in the case $I = (0,+\infty)$, when $\Delta_I$ is the punctured unit disc in the complex plane. Let $p_1: X \times \Delta_I \rightarrow X$ the first projection.
We then say that a collection $(\varphi_t)_{t \in I}$ of locally bounded K\"ahler potentials on $X$ is a \emph{subgeodesic ray} if $\Phi \in \mathrm{PSH}(X \times \Delta_I,p_1^*\omega)$, i.e.
$
p_1^*\omega + dd^c\Phi \geq 0
$
in the weak sense of currents. Moreover, it is said to be a 
\emph{geodesic ray} if it is subgeodesic and maximal with respect to this property, or equivalently, if the $S^1$-invariant associated function $\Phi$ satisfies the following homogeneous complex Monge-Amp\`ere equation 
$$
(\pi_1^*\omega + dd^c\Phi)^{n+1} = 0,
$$
on $\pi^{-1}(\bar{\Delta})$ seen as a manifold with boundary. We refer to e.g. \cite{Donaldsongeodesicequation, BBGZ}, and references therein, for details on this notion. For the delicate question of regularity of geodesic rays in this setting, see \cite{ChuTosattiWeinkove}.

\subsection{The K-energy functional and cscK metrics}

Let $(X,\omega)$ be a compact K\"ahler manifold, and let $$\mathrm{Ric}(\omega) := \sqrt{-1}\partial \bar{\partial}\log \omega^n$$ be the associated Ricci curvature form. We say that $\omega$ is a \emph{cscK metric} if it satisfies the cscK equation 
\begin{equation} \label{equation cscK}
\mathcal{S}(\omega) = \bar{\mathcal{S}},
\end{equation}
where $$\mathcal{S}(\omega):= \mathrm{tr}_{\omega} \mathrm{Ric}(\omega) := n \frac{\mathrm{Ric(\omega)} \wedge \omega^{n-1}}{\omega^n}$$ is the scalar curvature of $\omega$ and $\bar{\mathcal{S}}$ is the mean scalar curvature, given by
\begin{equation}
\bar{\mathcal{S}} := V^{-1} \int_X S(\omega) \; \omega^n =  n\frac{\int_X \mathrm{Ric}(\omega) \wedge \omega^{n-1}}{ \int_X \omega^n} := n  \frac{(c_1(X) \cdot  \alpha^{n-1})_X}{( \alpha^n)_X}.
\end{equation}
As observed by Mabuchi in \cite{MabuchiKenergy1} the cscK metrics can be characterized by variational methods, as the minima of a certain functional called the \emph{Mabuchi K-energy functional}. %Indeed, in his seminal paper \cite{MabuchiKenergy1} introduced the
It is the unique functional $M:  \mathrm{H}_{\omega} \rightarrow \mathbb{R}$ satisfying $\mathrm{M}(0) = 0$ and \begin{equation*} 
\frac{d}{dt} \mathrm{M}(\varphi_t) = - V^{-1} \int_X \dot{\varphi}_t(\mathcal{S}(\omega_{\varphi_t}) - \bar{\mathcal{S}})\; \omega_{\varphi_t}^n.
\end{equation*} 
for any smooth path $(\varphi_t)_{t \geq 0}$ in the $\mathcal{H}$. Note that part of the assertion of Mabuchi was that such a functional exists, and whenever they exist, the minimizers of this functional are precisely the cscK potentials $\varphi \in \mathcal{H}_{\omega}$, i.e. the corresponding K\"ahler form $\omega_{\varphi} := \omega + dd^c\varphi$ satisfies the cscK equation \eqref{equation cscK}. 

The K-energy can moreover be extended to the setting of locally bounded $\omega$-psh functions on $X$, i.e. to a functional $\mathrm{M}: \mathrm{PSH}(X,\omega) \cap L^{\infty}(X) \rightarrow \mathbb{R} \cup \{+\infty\}$. Similarily, there is an extension to the space $\mathcal{E}^1$ of locally finite energy potentials. To see this, recall that the K-energy functional can be written explicity using the so called Chen-Tian formula as the sum 
$$\mathrm{M} = \mathrm{M}_{pp} + \mathrm{M}_{ent}$$ 
of a pluripotential and an entropy part. Here
$$
\mathrm{M}_{ent}(\varphi) := V^{-1} \int_X \log \left(\frac{\omega_{\varphi}^n}{\omega^n}\right) \omega_{\varphi}^n \in [0,+\infty)
$$ 
and
$\mathrm{M}_{pp}(\varphi)$ is a linear combination of terms of the form 
$$
\int_X \varphi \omega^k \wedge \omega_{\varphi}^{n-k}
$$
and
$$
\int_X \varphi \omega^k \wedge \omega_{\varphi}^{n-k-1} \wedge \mathrm{Ric}(\omega).
$$
The pluripotential terms can be made sense of due to \cite{BBGZ, Darvas15}. The entropy term in the formula for $\mathrm{M}$ can always be made sense of as a lower semicontinous functional $\mathrm{M}_{ent}: \mathcal{E}^1 \rightarrow [0,+\infty]$, defined as the relative entropy of the probability measures $\omega_{\varphi}^n/V$ and $\omega^n/V$ (see \cite{ChenChengII}, \cite{Boucksomsurvey} and references therein).

\subsection{Holomorphic vector fields, the $\yen$-invariant and geodesic stability} \label{Section vector fields}

In order to define the notion of geodesic stability, we first introduce the notation for holomorphic vector fields that we will use: 
Suppose that $(X,\omega)$ is a compact K\"ahler manifold and denote by $J: TX \rightarrow TX$ the associated complex structure. A real vector field on $X$ is a section of the real tangent bundle $TX$ of $X$. It is said to be \emph{real holomorphic} if its flow preserves the complex structure, i.e. it has vanishing Lie derivative $L_VJ = 0$.
A holomorphic vector field on a \emph{compact} manifold is automatically $\mathbb{C}$-complete, and its flow $\phi_t$ is an action of $(\mathbb{C},+)$ on $X$ by holomorphic automorphisms. Conversely, one may associate to every additive action $\phi: \mathbb{C} \times X \rightarrow X$ by holomorphic automorphisms on $X$ the vector field 
$$
V_{\phi}(x) := \frac{d}{dt}  \phi(t,x)_{\vert t = 0},
$$
called the infinitesimal generator of $X$. The vector field $V_{\phi}$ is holomorphic and $\mathbb{C}$-complete on $X$, with the flow $\phi$. 

\begin{mydef}
A real holomorphic vector field $V$ on $X$ is said to be \emph{Hamiltonian} if it admits a \emph{Hamiltonian potential} $h_{\omega}^{V} \in C^{\infty}(X,\mathbb{R})$ such that the contraction 
$$
i_{V}(\omega) := V \rfloor \omega = \sqrt{-1} \bar{\partial} h_{\omega}^V .
$$
\end{mydef}

\begin{rem}
\emph
{Equivalently, a real holomorphic vector field admits a Hamiltonian potential if and only if it has a zero somewhere, see LeBrun-Simanca \cite{LeBrunSimanca}. }
\end{rem}  
\noindent 
Note further that the Hamiltonian potential is unique up to constants, so to relieve this ambiguity we impose the normalization
$$
\int_X h_{\omega}^V \omega^n = 0.
$$
{For the purpose of comparing with the situation for polarized manifolds $(X,L)$ it is interesting to recall that Hamiltonian vector fields are precisely those that lift to line bundles, see \cite[Lemma 12]{Donaldsontoric}. A real holomorphic Hamiltonian vector field is automatically a \emph{Killing field}, since $L_V J = L_V \omega = 0$ implies that also $L_V g = 0$ for the Riemannian metric associated to the K\"ahler form $\omega$.

\subsubsection{Geodesic rays induced by holomorphic vector fields}

For future use we recall also the notion of geodesic rays arising from holomorphic vector fields on $X$: 
%as was first remarked by Mabuchi \cite{Mabuchisymplectic}. 
In order to explain this notion, recall that the connected component of the Lie group $G := \mathrm{Aut}_0(X)$ of automorphisms of $X$ act on $\mathcal{K}$ by pullback $g.\omega := g^*\omega$, and induces a corresponding action on $\mathcal{H}_0$ via the identification $\mathcal{H}_0 \simeq \mathcal{K}$, as described in \eqref{equation action}. If $V$ is a real holomorphic Hamiltonian vector field on $X$, then $\exp(tJV)$ is an element of the Lie group $G$ for each $t \in [0,+\infty)$. If we set $$\omega_t := \exp(tJV)^*\omega, \; \; t \in [0,+\infty)$$ then $(\omega_t)_{t \geq 0}$ is a geodesic ray in $\mathcal{K}$, see \cite{Mabuchisymplectic}. The corresponding geodesic ray in $\mathcal{H}_0$ is denoted by $\varphi_t := \exp(tJV).\varphi_0$, where $\varphi_0 = 0$ such that $\omega_{\varphi_0} = \omega$.

\begin{mydef} \label{Definition induced vector field geodesic} A geodesic ray is said to be \emph{induced by the holomorphic vector field} if it is of the form $\varphi_t = \exp(tJV).\varphi_0$ for some real holomorphic Hamiltonian vector field $V$ on $X$, 
\end{mydef}

\subsubsection{The $\yen$-invariant and geodesic stability} \label{Section geodesic stability definition}

In order to state the definition of geodesic stability that we will use, let $[0,+\infty) \ni t \mapsto \rho(t) := \varphi_t$ be a given locally finite energy geodesic ray in $\mathcal{E}_0^1 := \mathcal{E}^1 \cap E^{-1}(0)$. 
Following \cite{ChenTang, ChenChengII, ChenChengIII} we consider the following numerical invariant 
$$
\yen(\rho(t)) := \lim_{t \rightarrow +\infty} t^{-1}\mathrm{M}(\rho(t)),
$$
associated to the given geodesic ray $\rho(t)$. 
This quantity is well-defined by convexity of the K-energy functional, see \cite{BB, ChenLiPaun} and also \cite{BDLconvexity} for convexity of the extension of the K-energy to finite energy spaces.
Recall also that two geodesic rays $(\varphi_t)$ and $(\xi_t)$ are said to be \emph{parallel} if 
$$d_{1,G}(\varphi_t, \xi_t) := \inf_{g,h \in G}d_1(g.\varphi_t, h.\xi_t) < C$$ for some constant $C > 0$ independent of $t$.

\begin{mydef} \label{Definition geodesic stability}
The pair $(X,[\omega])$ is said to be geodesically stable if and only if $\yen(\rho(t)) \geq 0$ for every unit speed geodesic ray $\rho(t)$, with equality precisely when $\rho(t)$ is induced by a holomorphic vector field on $X$. 
\end{mydef}

\noindent Note that in the paper \cite{ChenChengIII} of Chen-Cheng geodesic stability was defined with respect to rays parallel to geodesics induced by holomorphic vector fields, which we do not do here. However, our definition turns out to be equivalent to geodesic stability with respect to rays induced by holomorphic vector fields. Indeed, we have the following:

\begin{prop} \label{Prop equivalent conditions cscK assumption} \emph{(cf. \cite[Proposition 4.10]{SD2})}
Suppose that $(X,\omega)$ is a cscK manifold. Let $[0,+\infty) \ni t \mapsto \rho(t)$ be a unit speed geodesic ray in $\mathcal{E}_0^1$. Then the following are equivalent: 
\begin{enumerate}
    \item The ray $\rho(t)$ is of finite $d_{1,G}$-length, i.e. $d_{1,G}(\varphi_t, \varphi_0) < C$ for some constant $C > 0$ independent of $t$
    \item The ray $\rho(t)$ is parallel to a ray induced by a holomorphic vector field
    \item The ray $\rho(t)$ is itself induced by a holomorphic vector field
\end{enumerate}
\end{prop}

\begin{proof}
The implication $(1)$ and $(3)$ is precisely the statement of \cite[Proposition 4.10]{SD2}. The implication $(3) \Rightarrow (2)$ is immediate, since any ray is parallel to itself. Finally, if $\rho(t)$ is parallel to a ray $\xi_t^V$ induced by a holomorphic vector field (and with $\xi_0^V = \rho(0)$), then 
$$
d_{1,G}(\rho(t), \rho(0)) \leq d_{1,G}(\rho(t), \xi_t^V) + d_{1,G}(\xi_t^V, \rho(0)) < C  
$$
by the triangle inequality. Indeed, the first term is bounded by $C$ by the assumption that $\rho(t)$ and $\xi_t^V$ are parallel. Moreover 
$$
d_{1,G}(\xi_t^V, \rho(0)) = \inf_{g,h \in G} d_1(g.\xi_t^V, h.\rho(0)) = d_1(\rho(0), \rho(0)) = 0,
$$
since $xi_t^V = \exp(tJV).\rho(0)$ for some real holomorphic Hamiltonian vector field $V$ on $X$ (so in particular, $exp(tJV) \in G$). Hence $(2) \Rightarrow (1)$. 
Putting this together, we conclude that $(2) \Rightarrow (1) \Rightarrow (3) \Rightarrow (2)$, thus completing the proof. 
\end{proof}

\noindent As a corollary we have the following: 

\begin{cor}
Geodesic stability (Definition \ref{Definition geodesic stability}) is equivalent to geodesic stability in the sense of Chen-Cheng \cite{ChenChengIII}. 
\end{cor}

\begin{proof}
Geodesic stability in our sense clearly implies geodesic stability in the sense of Chen-Cheng \cite{ChenChengIII} (any ray is in particular parallel to itself). The latter geodesic stability condition was moreover proven in \cite{ChenChengIII} to be equivalent to existence of cscK metrics. On the other hand, suppose that $(X,\omega)$ is a cscK manifold. Then  it follows from Proposition \ref{Prop equivalent conditions cscK assumption} that a geodesic ray is parallel to a ray induced by a holomorphic vector field, if and only if it is itself induced by a holomorphic vector field. In other words, geodesic stability in the sense of Chen-Cheng implies the geodesic stability notion of Definition \ref{Definition geodesic stability}.
Putting this together, the considered geodesic stability notions must be equivalent. 
\end{proof}

\subsection{A weak variant of the geodesic stability notion} \label{Section Sgeodstability}

\noindent In order to later compare geodesic stability to K-stability notions (see Sections \ref{Section Kps as weak geod stab} and \ref{Section proof of main results}) it is also natural to introduce a slightly more flexible terminology. In this direction, we give the following definition, which emphasizes possible differences in the vanishing condition for the $\yen$-invariant: 

\begin{mydef} \emph{($(S,S_0)$-geodesic stability)}
Let $S_0 \subset S$ be subsets of the set of locally finite energy geodesic rays in $\mathcal{E}_0^1$. The pair $(X,[\omega])$ is then $(S,S_0)$-geodesically stable if and only if $\yen(\rho(t)) \geq 0$ for every unit speed geodesic ray $\rho(t) \in S$, with equality precisely when $\rho(t) \in S_0$. 
\end{mydef}

\noindent If $S$ is taken to be the full set of unit speed locally finite energy geodesic rays in $\mathcal{E}_0^1$, and $S_0$ is as any of the conditions $(1)-(3)$ in Proposition \ref{Prop equivalent conditions cscK assumption}, then $(S,S_0)$-geodesic stability of $(X,[\omega])$ is equivalent to geodesic stability of $(X,[\omega])$ (in the sense of Chen-Cheng, alternatively Definition \ref{Definition geodesic stability}).  We next recall the definitions of various stability notions in algebraic geometry, and show that they fit into the framework of the above notion of $(S,S_0)$-geodesic stability.

\bigskip

\section{Notions of K-polystability in K\"ahler geometry} \label{Section 3}

\noindent In this section we recall the general formalism of transcendental K-stability for K\"ahler manifolds, first introduced in \cite{SD1, DervanRoss}, and describe how various stability notions in algebraic geometry can be naturally defined from this point of view. 

\subsection{Preliminaries on test configurations}

We first recall the concept of test configurations for $X$, following \cite{SD1, SD2}. As a reference for this section we use \cite[Section 2.3]{SD2}. 

\begin{mydef}
A test configuration for $X$ consists of
\begin{itemize}
  \item \emph{a normal complex space $\mathcal{X}$, compact and K\"ahler, with a flat 
%\emph{(i.e. surjective)} 
morphism $\pi: \mathcal{X} \rightarrow \mathbb{P}^1$}
\item \emph{ a $\mathbb{C}^*$-action $\rho$ on $\mathcal{X}$ lifting the canonical action on $\mathbb{P}^1$ }
\item \emph{ a $\mathbb{C}^*$-equivariant isomorphism }
\begin{equation} \label{equiviso}
\mathcal{X} \setminus \pi^{-1}(0) \simeq X \times (\mathbb{P}^1 \setminus \{0\})
\end{equation}
\end{itemize}
\end{mydef}

\begin{rem}
\emph{Note that since $\pi$ is flat the central fiber $\mathcal{X}_0 := \pi^{-1}(0)$ is a Cartier divisor, so $\mathcal{X} \setminus \mathcal{X}_0$ is dense in $\mathcal{X}$ in Zariski topology. }
\end{rem}

\noindent The \emph{trivial} test configuration for $X$ is given by $(\mathcal{X} := X \times \mathbb{P}^1, \lambda_{\mathrm{triv}}, p_2)$, where $p_2: X \times \mathbb{P}^1 \rightarrow \mathbb{P}^1$ is the projection on the $2^{\mathrm{nd}}$ factor, and $\lambda_{\mathrm{triv}}: \mathbb{C}^* \times \mathcal{X} \rightarrow \mathcal{X}$, $(\tau, (x,z)) \mapsto (x, \tau z)$ is the $\mathbb{C}^*$-action that acts trivially on the first factor. If we instead let $\sigma: \mathbb{C}^* \times X \rightarrow X$ be any $\mathbb{C}^*$-action on $X$, then we obtain an induced test configuration as above with $\lambda(\tau, (x,z)) := (\sigma(\tau,x), \tau z)$ (by also taking the compactification so that the fiber at inifinity is trivial). Such test configurations are called \emph{product} test configurations of $(X,\alpha)$. In either case, we identify $X$ with $X \times \{ 1 \}$ and the canonical equivariant isomorphism \eqref{equiviso} is then explicitly induced by the isomorphisms $X \cong X \times \{1\} \rightarrow X \times \{ \tau \}$ given by $x \mapsto \lambda(\tau, (x,1)) =: \lambda(\tau) \cdot x$. Note moreover that if $V$ is any real holomorphic Hamiltonian vector field on $X$, then it may or may not generate a $\mathbb{C}^*$-action, and only if it does there is a clear way to associate a product test configuration to it (as described above). This is a subtle key issue.

We further define the notion of test configuration for $(X,\alpha)$, where $\alpha \in H^{1,1}(X,\mathbb{R})$ is any K\"ahler class on $X$: In order to do so, we first recall that the notions of K\"ahler forms and plurisubharmonic functions can be defined on complex spaces, see \cite{SD2, Zakthesis} for details on this in the present context. If $(X,L)$ is a polarized manifold, then a test configuration for $(X,L)$ is given by a $\mathbb{C}^*$-equivariant flat family $(\mathcal{X},\mathcal{L}) \rightarrow \mathbb{C}$, see e.g. \cite{BHJ1} and references therein for details and background on this classical definition first used in this form in \cite{Donaldsontoric}. More generally, we will work with the formalism for arbitrary K\"ahler manifolds, of which the above can be considered a special case. 
A test configuration for the polarized pair $(X,\alpha)$ is then defined as follows:

\begin{mydef}
A test configuration for $(X,\alpha)$ is a pair $(\mathcal{X},\mathcal{A})$ where $\mathcal{X}$ is a test configuration for $X$, and $\mathcal{A} \in H^{1,1}_{BC}(X,\mathbb{R})^{\mathbb{C}^*}$ is a  $\mathbb{C}^*$-invariant $(1,1)-$Bott-Chern cohomology class whose image under \eqref{equiviso}
is $p_1^*\alpha$. 
\end{mydef}

\noindent We give a few remarks and examples on how to compare cohomological test configurations with \emph{algebraic test configurations $(\mathcal{X},\mathcal{L})$ for a polarized manifold $(X,L)$}. 
\begin{enumerate}
\item If $(X,L)$ is any compact K\"ahler manifold endowed with an ample line bundle $L$ (so $X$ is projective), and $(\mathcal{X}, \mathcal{L})$ is a test configuration for $(X,L)$ in the usual algebraic sense, cf. e.g. \cite{Donaldsontoric}, then $(\mathcal{X}, c_1(\mathcal{L}))$ is a cohomological test configuration for $(X, c_1(L))$. This observation is useful, since many examples of algebraic test configurations $(\mathcal{X}, \mathcal{L})$ for polarized manifolds $(X,L)$ are known, see e.g. \cite{Gabornotes, Tianlecturenotes} and references therein.

\item There are more cohomological test configurations for $(X, c_1(L))$ than there are algebraic test configurations for $(X,L)$ (take for instance $(\mathcal{X}, \mathcal{A})$ with $\mathcal{A}$ a transcendental class as in the above definition), but in some cases the ensuing stability notions can nonetheless be seen to be equivalent (see \cite[Section 3]{SD1}).
\end{enumerate}

\subsection{Intersection theoretic numerical invariants}

Following \cite{SD1} we recall the following intersection theoretic definition of the classical Donaldson-Futaki invariant:

\begin{mydef} (\cite{SD1, DervanRoss}) \emph{To any cohomological test configuration $(\mathcal{X},\mathcal{A})$ for $(X,\alpha)$ we may associate its \emph{Donaldson-Futaki invariant} $\mathrm{DF}(\mathcal{X},\mathcal{A})$ and its \emph{non-Archimedean Mabuchi functional} ${{\mathrm{M}}^{\mathrm{NA}}}(\mathcal{X},\mathcal{A})$, first introduced in \cite{BHJ1}. They are given respectively by the following intersection numbers
\begin{equation*}
\mathrm{DF}(\mathcal{X},\mathcal{A}) := \frac{\bar{\mathcal{S}}}{n+1}V^{-1} (\mathcal{A}^{n+1})_{\hat{\mathcal{X}}} + V^{-1}(K_{\mathcal{X}/\mathbb{P}^1} \cdot \mathcal{A}^n)_{\hat{\mathcal{X}}}
\end{equation*}
and
$$
{{\mathrm{M}}^{\mathrm{NA}}}(\mathcal{X},\mathcal{A}) := \mathrm{DF}(\mathcal{X},\mathcal{A}) + ((\mathcal{X}_{0,\mathrm{red}} - \mathcal{X}_0) \cdot \mathcal{A}^n)_{\hat{\mathcal{X}}}
$$
computed on any smooth and dominating model $\tilde{\mathcal{X}}$ of $\mathcal{X}$ (due to the projection formula it does not matter which one). 
Note that $\mathrm{DF}(\mathcal{X},\mathcal{A}) \geq {{\mathrm{M}}^{\mathrm{NA}}}(\mathcal{X},\mathcal{A})$ with equality precisely when $\mathcal{X}_0$ is reduced. }
\end{mydef}

%\begin{rem} 
\noindent In case $\mathcal{X}$ is smooth, $K_{\mathcal{X}/\mathbb{P}^1} := K_{\mathcal{X}} - \pi^*K_{\mathbb{P}^1}$ denotes the relative canonical class taken with respect to the flat morphism $\pi:\mathcal{X} \rightarrow \mathbb{P}^1$. In the general case of a normal (possibly singular) test configuration $\mathcal{X}$ for $X$, we however need to give meaning to the intersection number $K_{\mathcal{X}} \cdot \mathcal{A}_1 \cdot \dots \cdot \mathcal{A}_n$, for $\mathcal{A}_i \in H^{1,1}_{\mathrm{BC}}(\mathcal{X},\mathbb{R})$. To do this, suppose that $\tilde{\mathcal{X}}$ is a smooth model for $\mathcal{X}$, with $\pi': \tilde{\mathcal{X}} \rightarrow \mathcal{X}$ the associated morphism. Since $\tilde{\mathcal{X}}$ is smooth the canonical class $K_{\tilde{\mathcal{X}}} := \omega_{\tilde{\mathcal{X}}}$ is a line bundle. Now consider 
$
\omega_{\mathcal{X}} := \mathcal{O}(K_{\mathcal{X}}) := (\pi'_*\omega_{\tilde{\mathcal{X}}})^{**},
$
i.e. the "reflexive extension" of $\omega_{\tilde{\mathcal{X}}}$, which is a rank $1$ reflexive sheaf on $\mathcal{X}$. We then set 
\begin{equation}
(\omega_{\mathcal{X}} \cdot \mathcal{A}_1 \cdot \dots \cdot \mathcal{A}_n) := ( K_{\tilde{\mathcal{X}}} \cdot \pi'^*\mathcal{A}_1 \cdot \dots \cdot \pi'^*\mathcal{A}_n).
\end{equation}
Using the projection formula (or an argument of the type \cite[Lemma 2.15]{DervanRoss} in the K\"ahler category) it is straightforward to see that the above intersection number is independent of the choice of model/resolution $\pi': \tilde{\mathcal{X}} \rightarrow \mathcal{X}$. In particular this holds for the Donaldson-Futaki invariant $\mathrm{DF}$ and the non-archimedean Mabuchi functional ${{\mathrm{M}}^{\mathrm{NA}}}$.

%\end{rem}

\subsection{Product test configurations and several definitions of K-polystability} \label{Section definitions Kps}

A number of natural variants of K-polystability for K\"ahler manifolds are given as follows: 

\begin{mydef} \label{Def Kps} \emph{In analogy with the usual definition for polarized manifolds, and following \cite[Section 3]{SD1}, we say that }
\begin{itemize}
\item \emph{$(X,\alpha)$ is \emph{K-semistable} if $\mathrm{DF}(\mathcal{X},\mathcal{A}) \geq 0$ for all normal and relatively K\"ahler test configurations $(\mathcal{X}, \mathcal{A})$ for $(X,\alpha)$. }
\item \emph{$(X,\alpha)$ is \emph{K-polystable} if it is K-semistable, and in addition $\mathrm{DF}(\mathcal{X},\mathcal{A}) = 0$ if and only if $\mathcal{X}$ is a product, where the latter means that one of the following conditions hold:}
\begin{enumerate}
    \item $\mathcal{X}_{\vert \pi^{-1}(\mathbb{C})}$
    \item $\mathcal{X}_0 := \pi^{-1}(0) \simeq X$
    \item $\mathcal{X}_{\vert \pi^{-1}(\Delta_r)} = X \times \Delta_r$, where $r \in (0,+\infty)$ and $\Delta_r := \{ z \in \mathbb{C} \; \vert \; |z| \leq r \}$.
    
\end{enumerate}
\emph{We will refer to these conditions as strong, weak and $r$-K-polystability respectively.}
\end{itemize}
\end{mydef}

\noindent Note that demanding that $\mathcal{X}$ is $\mathbb{C}^*$-equivariantly isomorphic to $X \times \mathbb{P}^1$ is not enough: For instance, there are (algebraic) product test configurations $(X,L) \times \mathbb{C}$ whose Donaldson-Futaki invariant vanishes, but whose compactifications over $\mathbb{P}^1$ (and thus their corresponding cohomological test configuration $(\bar{\mathcal{X}}, c_1(\bar{\mathcal{L}}))$) is \emph{not} a product. See e.g. \cite[Example 2.8]{BHJ1}. Hence the definition $(1)$ is the strongest notion of product that makes sense to consider in the context of K-polystability.

When it is necessary to make the distinction, we will refer to the above stability notions as \emph{cohomological}. In the same vein, we refer to the analogous stability notions for polarized manifolds (see e.g. \cite{Berman, BHJ1, Donaldsontoric}) as \emph{algebraic}. It is an interesting topic to compare cohomological and algebraic stability notions to eachother.

Regarding the cohomological notions, it was proven in \cite[Theorem A]{SD1} that cscK manifolds are always K-semistable. Moreover, if $(X,L)$ is a polarized manifold, then $(X,L)$ is K-semistable in the usual algebraic sense iff $(X, c_1(L))$ are (cohomologically) K-semistable (\cite[Proposition 3.14]{SD1}). In other words, the algebraic and the cohomological notions of K-semistability are equivalent. 
It is an open question whether the same holds for K-polystability, but at least one of the implications always holds: if $(X,L)$ is a polarized manifold such that $(X,c_1(L))$ is cohomologically K-polystable, then $(X,L)$ is algebraically K-polystable (cf. \cite[Proposition 2.22]{SD2}).
This holds regardless of the notion $(1) - (3)$ of product that one uses. In particular, the above notions of K-polystability generalizes the usual notion for polarized manifolds considered in \cite{Berman, BDL}.

\subsection{Test configurations embedded in the space of subgeodesic rays} \label{Section preliminary rays}

We here briefly recall a key notion from the papers \cite{SD1,SD2}, making precise the relationship between subgeodesic rays and test configurations, The goal is to view test configurations as ``embedded" in the space of subgeodesic rays on $X$, in a sense made precise below. This allows in particular to compare the $\yen$ and the Donaldson-Futaki invariants, and more generally, to interpret certain K-polystability notions as weak versions of geodesic stability, by restricting the set of rays along which one tests the $\yen$-invariant.

In order to recall the definition of subgeodesic rays \emph{compatible} with a given test configuration, we suppose that $(\mathcal{X},\mathcal{A})$ is a (possibly singular) relatively K\"ahler test configuration for $(X,\alpha)$. By taking the normalization of the graph of  $\mathcal{X} \dashrightarrow X \times \mathbb{P}^1$ and resolving singularities, we can always find a smooth model $\hat{\mathcal{X}}$ for $\mathcal{X}$, i.e. a a $\mathbb{C}^*$-equivariant bimeromorphic morphism $\rho: \hat{\mathcal{X}} \rightarrow \mathcal{X}$, where $\hat{\mathcal{X}}$ is smooth and dominates the product $X \times \mathbb{P}^1$. 
This yields the following situation:

\medskip

\[
 \begin{tikzcd}[ampersand replacement=\&]
    \hat{\mathcal{X}} \arrow[d, "\rho"] \arrow[ddr, controls={+(-1.5,-2) and +(0,0)},  "\pi", swap]
    \ar[dr, "\mu"]\\
    \mathcal{X} \arrow[dashed, r] \& X \times \mathbb{P}^1 \ar[r,"p_1"] \ar[d, "p_2"] \& X \\
    \& \mathbb{P}^1
 \end{tikzcd}
\]

\noindent Now let $(\varphi_t)$ be a locally bounded subgeodesic ray on $X$, with $\Phi$ the $S^1$-invariant function on $X \times \bar{\Delta}$ associated to the given ray $(\varphi_t)_{t \geq 0}$, such that $\varphi_t(x) = \Phi(x, e^{-t + is})$ for each $t \in [0,+\infty)$. By \cite[Proposition 3.10]{SD1} we then have  
$$
\rho^*\mathcal{A} = \mu^*p_1^*\alpha + [D],
$$
where $D = \sum_{j=1}^n a_i D_i$ is a divisor on $\hat{\mathcal{X}}$ supported on the central fiber $\hat{\mathcal{X}}_0$. We can further decompose the current of integration of $D$ as $\delta_D = \theta_D + dd^c\psi_D$, where $\theta_D$ is any smooth $S^1$-invariant $(1,1)$-form on $\hat{\mathcal{X}}$. Locally, we then have 
$$
\psi_D = \sum_j a_j \log|f_j| \; \; \textrm{mod} \; \mathcal{C}^{\infty}, 
$$ 
where the $f_j$ are local defining equations for the irreducible components $D_j$ respectively. 
Note that the choice of $\psi_D$ is uniquely determined modulo a smooth function on $\hat{\mathcal{X}}$, so in particular it determines a unique singularity type.
A locally bounded subgeodesic ray $(\varphi_t)_{t \geq 0}$ on $X$ is then said to be $L^{\infty}$-compatible with $(\mathcal{X}, \mathcal{A})$ if 
$
\Psi := \Phi \circ \mu + \psi_D
$
extends to a locally bounded $\rho^*\Omega$-psh function on $\hat{\mathcal{X}}$. 
{Similarily, a smooth curve $(\varphi_t)_{t \geq 0}$ is $C^{\infty}$-compatible with $(\mathcal{X}, \mathcal{A})$ if 
$
\Psi := \Phi \circ \mu + \psi_D
$
extends smoothly across $\hat{\mathcal{X}}_0$.}
In particular, an important point is that the singularity type of $\Phi \circ \mu$ is determined by the Green function $\psi_D$. 

\begin{ex}
\emph{We give two examples:} 
\begin{enumerate}

\item \emph{As a central example, let $\Omega$ be a smooth $S^1$-invariant $(1,1)$-form such that $[\Omega] = \mathcal{A}$. For $\tau \in (0,1]$ we denote by $\Omega_{\tau}$ the restriction of $\Omega$ to the fiber $\mathcal{X}_{\tau}$. As noted in \cite{DervanRoss}, $\Omega_{\tau}$ and $\Omega_1$ are cohomologous, so we may define a family of functions $(\varphi_{\tau})_{\tau \in (0,1]}$ on $X$ by the relation 
$
\lambda(\tau)^*\Omega_{\tau} - \Omega_1 = dd^c\varphi_{\tau}.
$
We can in turn define a $(\psi_t)_{t \in [0,+\infty)}$ on $X$ defined by the relation $\psi_t := \varphi_{e^{-t}}$. It is smooth and $C^{\infty}$-compatible with $(\mathcal{X},\mathcal{A})$, but not in general a subgeodesic ray (although it is still a useful tool in many cases, see e.g. \cite[Section 4]{SD1}). }

\item \emph{Moreover, there is a well known construction that yields a unique (up to certain choices) geodesic ray associated to a given test configuration, obtained by solving a certain homogeneous complex Monge-Amp\`ere equation on $\hat{\mathcal{X}}$. We refer to \cite[Section 4]{SD1} for details on the construction. This geodesic ray is then $L^{\infty}$-compatible but in general not $C^{\infty}$-compatible with $(\mathcal{X},\mathcal{A})$ (this relates to the intricate question of regularity of such geodesics, see e.g. \cite{ChuTosattiWeinkove}). }
\end{enumerate}
\end{ex}

\noindent The main results of \cite{SD2} consist of an injectivity lemma as well as results on asymptotics of energy functionals along compatible subgeodesic rays. They can be summarized by considering the assignment
$$
\mathrm{R}: (\mathcal{X}, \mathcal{A}) \mapsto [(\varphi_t)^{(\mathcal{X}, \mathcal{A})}],
$$
that maps any relatively K\"ahler test configuration $(\mathcal{X},\mathcal{A})$ for $(X,\alpha)$ to the set of subgeodesic rays compatible with $(\mathcal{X},\mathcal{A})$. This map satisfies the following two key properties: 

\begin{thm} \label{Theorem summary} \emph{(\cite[Theorem 1.5 and Theorem 1.8]{SD2})}
Suppose that $(X,\omega)$ is a compact K\"ahler manifold. Let $(\mathcal{X}, \mathcal{A})$ be a relatively K\"ahler test configuration for $(X,\alpha)$, and denote by $(\varphi_t)_{t \geq 0}$ any compatible subgeodesic ray in $\mathrm{R}(\mathcal{X},\mathcal{A})$. Then
\begin{enumerate}
    \item \emph{(Asymptotics of the K-energy)} 
$$
\lim_{t \rightarrow +\infty} t^{-1}\mathrm{M}(\varphi_t) = \mathrm{DF}(\mathcal{X},\mathcal{A}) + ((\mathcal{X}_{0,red} - \mathcal{X}_0) \cdot \mathcal{A}^n).
$$
    \item \emph{(Injectivity)} Let $(\mathcal{Y}, \mathcal{B})$ be another relatively K\"ahler test configuration for $(X,\alpha)$. Suppose that $$\mathrm{R}(\mathcal{X}, \mathcal{A}) \cap  \mathrm{R}(\mathcal{Y}, \mathcal{B}) \neq \emptyset.$$ Then the canonical $\mathbb{C}^*$-equivariant isomorphism $\mathcal{X} \setminus \mathcal{X}_0 \rightarrow \mathcal{Y}\setminus \mathcal{Y}_0$ extends to an isomorphism $\mathcal{X} \rightarrow \mathcal{Y}$. 
\end{enumerate}
\end{thm}

\noindent By relaxing the cscK assumption, we will improve on this result in Section \ref{Section proof of main results} below. 

\subsection{Interpretation of K-polystability as $(S,S_0)$-geodesic stability} \label{Section Kps as weak geod stab}

It is useful to take a point of view which promotes that K-stability is really something tested along geodesic rays, and we discuss how various K-stability notions can be viewed rather explicitly as special cases of the geodesic stability notion used in a recent series of remarkable papers by Chen-Cheng \cite{ChenChengI, ChenChengII, ChenChengIII}. The aim is thus to help clarifying the precise relationship between the abundance of stability notions available in the literature today. 
In order to do this, suppose that $\rho(t)$ is a locally finite energy unit speed geodesic ray in the space $\mathcal{E}_0^1 := \mathcal{E}^1 \cap E^{-1}(0)$. First recall the invariant
$$
\yen(\rho(t)) := \lim_{t \rightarrow +\infty} t^{-1}\mathrm{M}(\rho(t))
$$
introduced in \cite{ChenChengI, ChenChengII, ChenChengIII} (here $\mathrm{M}$ is the Mabuchi K-energy functional). If the geodesic ray $\rho(t)$ is "compatible" with a test configuration $(\mathcal{X}, \mathcal{A})$ for $(X,[\omega])$, in a sense made precise in \cite{SD1, SD2}, 
then the $\yen$-invariant essentially coincides with the Donaldson-Futaki invariant (up to an explicit error term that vanishes when the total space of the test configuration is reduced). 
More precisely, we have
$$
\yen(\rho(t)) = \mathrm{DF}(\mathcal{X},\mathcal{A}) + ((\mathcal{X}_{0,red} - \mathcal{X}_0) \cdot \mathcal{A}^n),
$$
by Theorem \ref{Theorem summary}. 
With reference to the definition of $(S,S_0)$-geodesic stability introduced in Section \ref{Section Sgeodstability}, recall that when $S$ is taken to be the set of \emph{all} unit speed geodesic rays in $\mathcal{E}_0^1$, and $S_0$ is the set of all geodesic rays induced by holomorphic vector fields on $X$, then $(S,S_0)$-geodesic stability turns out to be equivalent to the geodesic stability notion used in \cite{ChenChengII, ChenChengIII}. When it comes to geodesic K-polystability, we have an analogous interpretation as follows:

\begin{thm}
The pair $(X,[\omega])$ is geodesically K-polystable if and only if it is $(S,S_0)$-geodesically stable, where $S$ is the set of subgeodesic rays compatible with a relatively K\"ahler test configuration for $(X,[\omega])$ and $S_0$ is the set of geodesic rays induced by some holomorphic vector field on $X$.
\end{thm}

\begin{rem} \label{Remark comparison stability}
\emph{A similar result applies also to other stability notions, such as \emph{slope stability}, introduced in \cite{RossThomas}. More precisely, one then chooses $S$ as the set of all subgeodesic rays compatible with test configurations given by the deformation to the normal cone construction (see e.g. \cite[Example 5.3]{Bermantransc}). Moreover, all the alternative K-polystability notions discussed in Section \ref{Section definitions Kps} are also of this form, by varying the set $S_0$ in the obvious way (the set of subgeodesic rays compatible with product test configurations, in the various senses respectively).}
\end{rem} 

\noindent The above discussion fits well with the well known connection between test configurations and geodesic rays, as well as geodesic stability and the Yau-Tian-Donaldson conjecture. The notation introduced may also serve as a convenient common framework for all these different stability notions in K\"ahler geometry.

\bigskip

\section{Proof of main results} \label{Section proof of main results}
In this section we state and prove our main results, relying on the idea of considering special paths of test configurations when varying the underlying K\"ahler class. 

\subsection{Stability loci in the K\"ahler cone}

Let $\mathcal{C}_X \subset H^{1,1}(X,\mathbb{R})$ be the open cone of K\"ahler classes on $X$. The classical point of view on the question of existence of canonical metrics is to exploit a variational approach, when it is natural to characterize existence of constant scalar curvature K\"ahler metrics \emph{in a given K\"ahler class} on a given compact K\"ahler manifold. Thanks to the introduction of stability notions for pairs $(X,\alpha)$ of a given K\"ahler manifold and a K\"ahler class $\alpha \in \mathcal{C}_X$, we may however ask the following question: Given a compact K\"ahler manifold $(X,\omega)$, can we characterize the subsets of $\mathcal{C}_X$ consisting of the K\"ahler classes $\alpha$ for which the pair $(X,\alpha)$ is K-polystable/geodesically K-stable/cscK. The same question can of course be asked for any stability condition (K-semistability, slope stability etc). This slight change in point of view is sometimes useful, as we show below. In particular, note that the Yau-Tian-Donaldson conjecture can be reformulated as the statement that the cscK locus (alternatively, the geodesically stable locus) equals to K-polystable locus (in a suitable sense).

From the work of Berman-Darvas-Lu \cite{BDL} and Chen-Cheng \cite{ChenChengIII} we have equality of the cscK locus and the geodesically stable locus, and from the work \cite{SD1,DervanRoss, SD2, Dervanrelative} we have inclusions
$$
\textrm{cscK locus} \subseteq \textrm{geodesically K-polystable locus} \subseteq \textrm{K-semistable locus},
$$
and in \cite[Appendix]{SD2} it was proven that the $$\textrm{geodesically K-polystable locus} \subseteq \textrm{equivariantly K-polystable locus}.$$ 

\noindent The various K-polystability notions discussed in Section \ref{Section definitions Kps} have similar inclusions, but it is an open question whether equality holds (we will show in Theorem \ref{Theorem equivalence Kps notions} below that this is indeed the case). When it comes to one of the main questions of this paper, namely comparing K-polystability with geodesic K-polystability,  we know that there is an inclusion
$$
\textrm{K-polystable locus} \subseteq \textrm{geodesically K-polystable locus}.
$$
However, it is quite possible that for certain unquantized (i.e. unpolarized) compact K\"ahler manifolds the K-polystable locus is in fact empty (using the stronger transcendental stability notion, see Section \ref{Section definitions Kps}. More precisely, the following questions are of particular interest to us here: 
\begin{question}
Do we have an inclusion $\textrm{cscK locus} \subseteq \textrm{K-polystable locus}$? Do the K-polystable and geodesically K-polystable loci coincide in general? 
\end{question}

\noindent By the above discussion, an affirmative answer to the second question implies an affirmative answer also to the first one. In the sections that follow we will develop the tools to state and prove some partial results in this direction.

\subsection{Special paths of test configurations when varying the underlying K\"ahler class}
%Proof of Theorem \ref{Theorem convex combination}} 
We first focus on relating test configurations to eachother in the case when we change also the underlying K\"ahler class. More precisely, we set out to compare test configurations for $(X,\alpha)$ and $(X,\beta)$, where $\alpha, \beta \in \mathcal{C}_X$ are different K\"ahler classes on $X$, which in turn yields a proof of Theorem \ref{Theorem convex combination}. As a first observation, we note the following result on the $\yen$-invariant under convex combinations of rays along convex combinations of the underlying K\"ahler classes:  

\begin{prop}
Let $\alpha, \beta \in \mathcal{C}_X$ and set $\alpha_s := (1-k)\alpha + s\beta$, for $s \in [0,1]$. Suppose that $\rho_{\alpha}(t)$ and $\rho_{\beta}(t)$ are smooth subgeodesic rays with respect to $(X,\alpha)$ and $(X,\beta)$ respectively. Then $\rho_s(t) := (1-s)\rho_{\alpha}(t) + s \rho_{\beta}(t)$ are subgeodesic rays with respect to $(X,\alpha_s)$, and the map $$[0,1] \ni s \mapsto \yen(\rho_s(t))$$ is continuous.
\end{prop}

\begin{proof}
This follows immediately from the definitions. Indeed, fix K\"ahler forms $\omega_0, \omega_1$ on $X$ such that $\alpha := [\omega_0]$ and $\beta := [\omega_1]$. In turn, let $\omega_s := (1-s)\omega_0 + s\omega_1$, such that $\alpha_s = [\omega_s]$. By hypothesis we then have 
$$
\omega_0 + dd^c\rho_{\alpha}(t) \geq 0
$$
and
$$
\omega_1 + dd^c\rho_{\beta}(t) \geq 0.
$$
Therefore also
$$
\omega_s + dd^c\rho_{s}(t) = (1-s)(\omega_0 + dd^c\rho_{\alpha}(t)) + s(\omega_1 + dd^c\rho_{\beta}(t)) \geq 0,
$$
i.e. $\rho_s(t)$ is a subgeodesic ray. Finally, the map $$[0,1] \ni s \mapsto \yen(\rho_s(t))$$ is clearly a polynomial in $s$, hence continuous. 
\end{proof}

\noindent As a particular case we obtain the following result on convex combinations of test configurations, where parts $(1)-(2)$ should be compared with an observation in A. Isopoussu's thesis \cite{Antoniso} (where only the setting of polarized manifolds was considered):

\begin{thm}
Let $\alpha, \beta \in \mathcal{C}_X$ and set $\alpha_s := (1-s)\alpha + s\beta$, for $s \in [0,1]$. Suppose that $(\mathcal{X},\mathcal{A})$ and $(\mathcal{X},\mathcal{B})$ is a relatively K\"ahler test configuration for $(X,\alpha)$ and $(X,\beta)$ respectively. Then
\begin{enumerate}
    \item $(\mathcal{X}, (1-s)\mathcal{A} + s \mathcal{B})$ is a relatively K\"ahler test configuration for $(X,\alpha_s)$.
    \item The maps $$[0,1] \ni s \mapsto \mathrm{DF}(\mathcal{X}, (1-s)\mathcal{A} + s \mathcal{B})$$ and $$[0,1] \ni s \mapsto \mathrm{J}^{\mathrm{NA}}(\mathcal{X}, (1-s)\mathcal{A} + k \mathcal{B})$$ are continuous.
    %\footnote{In fact, they are polynomials in $k$ of degree at most $n+1$.}
    \item Suppose that $\rho_A(t)$ and $\rho_B(t)$ are subgeodesic rays $C^{\infty}$-compatible with $(\mathcal{X},\mathcal{A})$ and $(\mathcal{X},\mathcal{B})$ respectively, and write $\rho_s(t) := (1-s)\rho_A(t) + k \rho_B(t)$ If $\alpha_s = [\omega_s]$, then 
    $$
    \mathrm{DF}(\mathcal{X}, (1-s)\mathcal{A} + s \mathcal{B}) = \lim_{t \rightarrow +\infty} t^{-1} \mathrm{M}_{\omega_t}(\rho_s(t)) - ((\mathcal{X}_{0,red} - \mathcal{X}_0) \cdot \mathcal{A}^n). 
    $$
\end{enumerate}
\end{thm}

\begin{proof}
The first assertion follows from the basic fact that the set of relatively K\"ahler classes on $\mathcal{X}$ is convex. 
In order to see that $(\mathcal{X}, (1-s)\mathcal{A} + s \mathcal{B})$ is a test configuration for $(X,\alpha_s)$ we may pass to a resolution $\rho: \hat{\mathcal{X}} \rightarrow \mathcal{X}$. Then
$
\rho^*\mathcal{A} = \mu^*p_1^*\alpha + [D]
$
and
$
\rho^*\mathcal{B} = \mu^*p_1^*\alpha + [E].
$
Hence
$$
\rho^*\left((1-s)\mathcal{A} + s \mathcal{B}\right) = \mu^*p_1^*\alpha_s + (1-s)[D] + s[E],
$$
and the conclusion $(1)$ follows. The assertion $(2)$ follows immediately from the definition of $\mathrm{DF}$ and $\mathrm{J}^{\mathrm{NA}}$ as intersection numbers (it is straightforward to see that $[0,1] \ni s \mapsto \mathrm{DF}(\mathcal{X}, (1-s)\mathcal{A} + s \mathcal{B})$ and $[0,1] \ni s \mapsto \mathrm{J}^{\mathrm{NA}}(\mathcal{X}, (1-s)\mathcal{A} + s \mathcal{B})$ are polynomials in $k$ of degree at most $n+1$, thus continuous). 
Finally, in order to prove $(3)$ it suffices (by \cite[Theorem 1.5]{SD2} to show that $\rho_s(t)$ is $C^{\infty}$-compatible with $(\mathcal{X}, (1-s)\mathcal{A} + s \mathcal{B})$. This is also immediate. To see it, let $\Phi_0$ and $\Phi_1$ denote the $S^1$-invariant functions on $X \times \bar{\Delta}$ associated to $\rho_{\alpha}(t)$ and $\rho_{\beta}(t)$ respectively (so that $\rho_{\alpha}(t) = \Phi_0(x,e^{-t+iv})$ and $\rho_{\beta}(t) = \Phi_1(x,e^{-t+iv})$ as before). If we let $\mu, [E], [D]$ be as above, then
$$
\mu^*((1-s)\Phi_0 + s\Phi_1) + (1-s)[D] + s[E] = 
(1-s)(\mu^*\Phi_0 + [D]) + s(\mu^*\Phi_1 + [E])
$$
extends smoothly across $\mathcal{X}_0$ (since both terms do so, by hypothesis). This concludes the proof. 
\end{proof}

\noindent An interesting point is to emphasize that these proofs become very simple once we take the point of view chosen above.

\subsection{The set of product configurations}

It is a subtle but important point to understand how to properly define the concept of a product test configuration. A suggestion in \cite{SD1} was that a test configuration should be called a product (or "geodesic product") if and only if it is compatible with a geodesic ray induced by a holomorphic vector field. Of course, if $(\mathcal{X},\mathcal{A})$ is a product in the traditional sense (i.e. the total space is isomorphic to $X \times \mathbb{C}$ away from the fiber at infinity), then it is compatible with a ray of this form. The more difficult part is to establish the converse, in which case of only partial results are known: First, if we restrict to the case of polarized manifold $(X,L)$ and their usual algebraic test configurations $(\mathcal{X},\mathcal{L})$, then it was proven in \cite{BDL} that this holds whenever the underlying class is cscK. Secondly, assuming existence of a cscK metric, the same holds for the more general transcendental test configurations $(\mathcal{X},\mathcal{A})$ for $(X,\alpha)$, provided that the test configuration is taken to be \emph{equivariant} (see \cite{Dervanrelative, SD2}). 

The goal of this section is to explain that the hypothesis that the underlying class is cscK can be weakened. Indeed, we will show that it is enough to assume that there exists a cscK metric in some (possibly different) K\"ahler class on $X$, i.e. the $\textrm{cscK locus} \neq \emptyset$. In order to establish this result, the following lemma constitutes the key step:  

\begin{lem}  \label{Lemma main}
Let $\alpha, \beta \in \mathcal{C}_X$ such that also $\beta - \alpha \in \mathcal{C}_X$. Suppose that $(\mathcal{X},\mathcal{A})$ is a relatively K\"ahler test configuration for $(X,\alpha)$, with associated geodesic ray $\rho(t)$. Then there exists a relatively K\"ahler test configuration $(\mathcal{X},\mathcal{B})$ for $(X,\beta)$, with the same total space $\mathcal{X}$, and which is $C^{\infty}$-compatible with the same geodesic ray $\rho(t)$.
\end{lem} 

\begin{proof}
By resolution of indeterminacy there is a smooth and dominating test configuration $\rho: \hat{\mathcal{X}} \rightarrow \mathcal{X}$ for $X$ such that 

\[
 \begin{tikzcd}[ampersand replacement=\&]
    \hat{\mathcal{X}} \arrow[d, "\rho"] 
    \ar[dr, "\mu"]\\
    \mathcal{X} \arrow[dashed, r] \& X \times \mathbb{P}^1 \ar[r,"p_1"] \& X \\
 \end{tikzcd}
\]

\noindent Then it follows from \cite[Proposition 3.10]{SD1} that 
$
\rho^*\mathcal{A}_{\alpha} = \mu^*p_1^*\alpha + [D]
$
for some $\mathbb{R}$-divisor $D$ on $\hat{\mathcal{X}}$ supported on $\hat{\mathcal{X}}_0$. 
Now set $\mathcal{A}_{\beta} := \mathcal{A} + \eta$, where $\eta$ is a $(1,1)$-cohomology class on $\mathcal{X}$ which satisfies
$$
\rho^*\mathcal{A}_{\beta} - \rho^*\mathcal{A}_{\alpha} = \mu^*p_1^*(\beta - \alpha).
$$
Since $\beta - \alpha \in \mathcal{C}_X$, it follows that $\mu^*p_1^*(\beta - \alpha)$ is nef. Therefore $\rho^*\eta$ is nef on $\hat{\mathcal{X}}$, so also $\eta$ is nef on $\mathcal{X}$, and  
$
\mathcal{A}_{\beta} = \mathcal{A} + \eta
$
is relatively K\"ahler (as a sum of a relatively K\"ahler and relatively nef classes). Hence we have a cohomological test configuration $(\mathcal{X}, \mathcal{A}_{\beta})$ for $(X,\beta)$. Moreover, this new test configuration satisfies
$
\rho^*\mathcal{A}_{\beta} = \mu^*p_1^*\beta + [D],
$
with the same $\mu$ and $[D]$ as before. As a consequence, also $(\mathcal{X}, \mathcal{A}_{\beta})$ is $C^{\infty}$-compatible with the geodesic ray $\rho(t)$, which is what we wanted to prove.
\end{proof}

\noindent In particular we then obtain the following corollary, where $\mathrm{Fut}_{\alpha}(X,V)$ denotes the classical Futaki invariant of the vector field $V$ on $X$, and $\mathcal{C}_F \subseteq \mathcal{C}_X$ denotes the set of all K\"ahler classes $\alpha$ for which $\mathrm{Fut}_{\alpha}(X, \cdot)$ vanishes identically:

\begin{prop} \label{Prop changing class for products}
Suppose that $(\mathcal{X}, \mathcal{A})$ is a relatively K\"ahler test configuration for $(X,\alpha)$ whose associated geodesic ray $\rho(t)$ is induced by a holomorphic vector field $V$ on $X$. Then for each $\beta \in \mathcal{C}_X$ there is a relatively K\"ahler test configuration $(\mathcal{X},\mathcal{A}_{\beta})$ for $(X, \beta)$, with the same total space $\mathcal{X}$, such that 
$$
\mathrm{DF}(\mathcal{X},\mathcal{A}_{\beta}) = \mathrm{Fut}_{\beta}(X,V).
$$
In particular, if $\beta \in \mathcal{C}_F$, then $\mathrm{DF}(\mathcal{X}, \mathcal{A}_{\beta}) = 0$.
\end{prop}

\begin{proof}
Pick $\lambda > 0$ such that $\lambda \beta - \alpha \in \mathcal{C}_X$. By Lemma \ref{Lemma main} there exists a relatively K\"ahler test configuration $(\mathcal{X},\mathcal{A}_{\lambda \beta})$ for $(X, \lambda \beta)$ that is compatible with a ray $\rho(t)$ induced by a holomorphic vector field, i.e. of the form $\rho(t) := \exp(tJV).\rho(0)$ for some real holomorphic Hamiltonian vector field $V$ on $X$. By \cite[Theorem 3.10]{SD2} we then have
$$
\mathrm{DF}(\mathcal{X},\mathcal{A}_{\lambda \beta}) = \lim_{t \rightarrow +\infty} \frac{d}{dt} \mathrm{M}(\rho(t)) =  \mathrm{Fut}_{\lambda \beta}(X,V), 
$$
which vanishes if $\beta \in \mathcal{C}_F$
Finally, set $\mathcal{A}_{\beta} := \lambda^{-1} \mathcal{A}_{\lambda\beta}$. Then $(\mathcal{X},\mathcal{A}_{\beta})$ is a relatively K\"ahler test configuration for $(X,\beta)$ and 
$$
\bar{S}_{\lambda\beta} = \lambda^{-1}\bar{S}_{\beta}
$$
$$
V_{\lambda \beta} = \lambda^n V_{\beta}.
$$
One then checks that
$$
\mathrm{DF}_{\lambda\beta}(X,\lambda \mathcal{A}_{\beta}) = \frac{\bar{S}_{\lambda \beta} }{(n+1)V_{\lambda \beta}} (\lambda \mathcal{A}_{\beta})^{n+1} + \frac{1}{V_{\lambda \beta}} (K_{\mathcal{X}/\mathbb{P}^1} \cdot (\lambda\mathcal{A}_{\beta})^n) =
$$
$$
= \frac{\bar{S}_{\beta} }{(n+1)V_{\beta}} (\mathcal{A}_{\beta})^{n+1} + \frac{1}{V_{ \beta}} (K_{\mathcal{X}/\mathbb{P}^1} \cdot \mathcal{A}_{\beta}^n) = \mathrm{DF}_{\beta}(X,\mathcal{A}_{\beta}),
$$
which in turn equals $\mathrm{Fut}_{\beta}(X,V)$.
This completes the proof.
\end{proof}

\noindent This result has several immediate applications as explained below. 

\subsection{Proof of Theorem \ref{Theorem second main} and Main Theorem \ref{Theorem dichotomy main}}

As a first application of Lemma \ref{Lemma main} and Proposition \ref{Prop changing class for products}, we prove the following (the main new point being that we allow changing the underlying K\"ahler class):

\begin{thm}
Let $\alpha \in \mathcal{C}_X$ and suppose that $(\mathcal{X},\mathcal{A})$ is a relatively K\"ahler smooth and dominating test configuration for $(X,\alpha)$. Then, for each $\beta \in \mathcal{C}_X$ there is a $\lambda > 0$ such that $\lambda \beta > \alpha$, and a relatively K\"ahler test configuration $(\mathcal{Y}, \mathcal{B})$ for $(X,\lambda \beta)$ such that 
\begin{enumerate}
    \item $\mathcal{Y} = \mathcal{X}$,
    \item The test configurations $(\mathcal{X},\mathcal{A}) \sim (\mathcal{Y}, \mathcal{B})$, i.e. there is a subgeodesic ray $\rho(t)$ that is $C^{\infty}$-compatible with both. 
\end{enumerate}
In particular, if $\alpha = [\omega]$ and $\lambda \beta = [\theta]$, then we have 
    $$
    \mathrm{DF}(\mathcal{X}, \mathcal{A}) = \lim_{t \rightarrow +\infty} t^{-1} \mathrm{M}_{\omega}(\rho(t)) - ((\mathcal{X}_{0,red} - \mathcal{X}_0) \cdot \mathcal{A}^n)
    $$
    and
    $$
     \mathrm{DF}(\mathcal{X}, \mathcal{B}) = \lim_{t \rightarrow +\infty} t^{-1} \mathrm{M}_{\theta}(\rho(t)) - ((\mathcal{X}_{0,red} - \mathcal{X}_0) \cdot \mathcal{B}^n).
    $$
\end{thm}

\begin{proof}
The statements $(1)$ and $(2)$ are simply Lemma \ref{Lemma main}. The last statement regarding the asymptotics of the K-energy is precisely \cite[Theorem 1.5]{SD2}.  
\end{proof}

\noindent The first main point of the above discussion is that we may now deduce the following main result: 

\begin{thm} \label{Theorem dichotomy not intro} \emph{(cf. Theorem \ref{Theorem dichotomy})}
Let $(X,\omega)$ be a compact K\"ahler manifold and suppose that the K-polystable locus $\neq \emptyset$. Suppose that $(\mathcal{X}, \mathcal{A})$ is a test configuration for $(X,[\omega])$. Then the following are equivalent:
\begin{itemize}
\item$\mathcal{X}_{\pi^{-1}(\mathbb{C})} \simeq X \times \mathbb{C}$ 
%\item $\mathcal{X}_0 \simeq X$ \textbf{(???)}
\item The associated geodesic ray is induced by a holomorphic vector field on $X$. 
\end{itemize}
\end{thm}

\begin{proof}
Since the K-polystable locus $\subseteq \mathcal{C}_F$ we may without loss of generality assume that $\alpha := [\omega] \in \mathcal{C}_F$. 
Now suppose for contradiction that the K-polystable locus is \emph{strictly} contained in the geodesically K-polystable locus. Then there is a relatively K\"ahler test configuration $(\mathcal{X},\mathcal{A})$ which is a geodesic product (i.e. $C^{\infty}$-compatible with a subgeodesic ray induced by a real holomorphic Hamiltonian vector field $V$ on $X$), but not a product configuration (in the sense that $\mathcal{X}_{\pi^{-1}(\mathbb{C})} \simeq X$). Assuming that $\alpha \in \mathcal{C}_F$ we then have $$\mathrm{DF}(\mathcal{X},\mathcal{A}) = \mathrm{Fut}_{\alpha}(X,V) = 0.$$ Moreover, the K-polystable locus is non-empty, so we may pick $\beta \in \mathcal{C}_X$ in such a way so that $(X,\beta)$ is K-polystable. By Proposition \ref{Prop changing class for products} there is then a test configuration $(\mathcal{X},\mathcal{B})$ for $(X,\beta)$, with the same total space $\mathcal{X}$, such that $\mathrm{DF}(\mathcal{X},\mathcal{B}) = 0$ (indeed $(\mathcal{X},\mathcal{B})$ is a geodesic product and $\beta \in \mathcal{C}_F$ because $(X,\beta)$ is K-polystable). 
Since $(\mathcal{X},\mathcal{B})$ is a relatively K\"ahler \emph{non}-product configuration, this contradicts that $(X,\beta)$ is K-polystable. Hence, if the K-polystable locus is non-empty then it must coincide with the geodesically K-polystable locus. In particular, the conditions $(1)$ and $(2)$ are equivalent. This finishes the proof.
\end{proof}

\noindent In particular, the above proof gives a partial answer to the question of comparing the K-polystability and geodesic K-polystability notions: 

\begin{cor} 
Let $(X,\omega)$ be a compact K\"ahler manifold and suppose that the K-polystable locus $\neq \emptyset$. Then the K-polystable locus equals the geodesically K-polystable locus. 
\end{cor}

\noindent As a next key point, the above results are independent of whether we consider K-polystability with respect to $\mathcal{X}_{\pi^{-1}(\mathbb{C})} \simeq X \times \mathbb{C}$ or $\mathcal{X}_0 \simeq X$, as explained below.

\subsection{Equivalence of notions of product configuration}

We now discuss the equivalence of various notions of product configurations and their corresponding K-polystability notions. For the purpose of this discussion, consider the following list of reasonable variants of the usual algebraic notion of product configuration: 

\begin{mydef}
We say that $(X,\alpha)$ is 
\begin{enumerate}
    \item strong K-polystable if it is K-polystable with respect to product configurations in the sense that $\mathcal{X}_{\pi^{-1}(\mathbb{C})} \simeq X \times \mathbb{C}$.
    \item weak K-polystable if it is K-polystable with respect to product configurations in the sense that $\mathcal{X}_0 \simeq X$. 
    \item $r$-K-polystable if it is K-polystable with respect to product configurations in the sense that $\mathcal{X}_{\pi^{-1}(\Delta_r)} \simeq X \times \Delta_r$, for any $r > 0$.
\end{enumerate}
\end{mydef}

\begin{rem}
\emph{In the case of polarized manifolds $(X,L)$ the strong K-polystability condition is rather that $\mathcal{X} \simeq X \times \mathbb{C}$, since then test configurations are usually defined over $\mathbb{C}$ rather than directly over $\mathbb{P}^1$. }
\end{rem}

\noindent The strong and weak K-polystability notions are both used frequently in the literature surrounding the YTD conjecture, see e.g. \cite{Boucksomsurvey} and references therein. The goal is now to seize the opportunity to address the question of whether or not these conditions $(1)-(3)$ are in fact equivalent. As preparation, we first check the following simple claim, suggested by the terminology: 

 \begin{prop} \label{Prop strong implies weak}
 If $r > r'$ then strong K-polystability $\Rightarrow$ $r$-K-polystability $\Rightarrow$ $r'$-K-polystability $\Rightarrow$ weak K-polystability
 \end{prop} 
 
\begin{proof}
Suppose that there is a K\"ahler class $\alpha \in \mathcal{C}_X$ which is a strongly K-polystable but not weakly K-polystable. Then there is a test configuration satisfying $\mathcal{X}_0 \simeq X$, $\mathcal{X}_{\pi^{-1}(\mathbb{C})} \not\simeq X \times \mathbb{C}$ and $\mathrm{DF}(\mathcal{X},\mathcal{A}) > 0$. But this is a contradiction. The same argument goes through if $r > r'$, since then $\mathcal{X}_{\pi^{-1}(\Delta_r)} \simeq X \times \Delta_r$ implies that $\mathcal{X}_{\pi^{-1}(\Delta_{r'})} \simeq X \times \Delta_{r'}$.
\end{proof}
 
 \noindent We now address the question of whether these a priori differing K-polystability notions, used by various authors in the literature, are in fact equivalent. Conveniently, it turns out that this is the case, thus clarifying the relationship between various results regarding the respective notions of K-polystability:
 
\begin{thm} \label{Theorem equivalence Kps notions}
Suppose that $(X,\omega)$ is a compact K\"ahler manifold with non-empty strong K-polystability locus. Then the following notions are equivalent: 
\begin{enumerate}
    \item Strong K-polystability
    \item $r$-K-polystability for any $r \in (0,+\infty)$
    \item Weak K-polystability
    \item Geodesic K-polystability
    \item $S$-geodesic stability with respect to the set of all geodesic rays compatible with relatively K\"ahler test configurations for $(X,\alpha)$
\end{enumerate}
\end{thm}

\begin{proof}

\noindent Let $S$ be a subset of all relatively K\"ahler test configurations for $X$. Suppose for contradiction that there is an $\alpha \in \mathcal{C}_X$ such that $(X,\alpha)$ is weakly K-polystable but not strongly K-polystable. Then there is, by definition, a test configuration $(\mathcal{X}, \mathcal{A})$ for $(X,\alpha)$ which is relatively K\"ahler, and satisfies $\mathcal{X}_0 \simeq X$, $\mathrm{DF}(\mathcal{X},\mathcal{A}) = 0$, but $\mathrm{X}_{\pi^{-1}(\mathbb{C})} \not \simeq X \times \mathbb{C}$. Now pick $\beta$ strongly K-polystable. Then $(\mathcal{X}, \mathcal{A} + \mu^*p_1^*(\beta - \alpha))$ is a relatively K\"ahler test configuration for $(X,\beta)$, with the same total space as $(\mathcal{X},\mathcal{A})$. But by Proposition \ref{Prop strong implies weak} the pair $(X,\beta)$ is, in particular, weakly K-polystable, and $\mathrm{X}_0 \simeq X$. Hence, by definition, $\mathrm{DF}(\mathcal{X}, \mathcal{A} + \mu^*p_1^*(\beta - \alpha)) = 0$. Finally, since $(X,\beta)$ is also strongly K-polystable, we then have $\mathrm{X}_{\pi^{-1}(\mathbb{C})} \simeq X \times \mathbb{C})$. Conversely, it is clear that if $\mathrm{X}_{\pi^{-1}(\mathbb{C})} \simeq X \times \mathbb{C})$ then also $\mathcal{X}_0 \simeq X$. This finishes the proof of the equivalence $(1) \Leftrightarrow (3)$.

The exact same proof applies to any situation when we compare K-polystability notions with respect to notions of product where one notion implies the other, and both satisfy the requirement that $\mathrm{DF}(\mathcal{X},\mathcal{A}) = 0$ for products. This way we prove that $(1) \Leftrightarrow (2)$. Finally, the equivalence $(1) \Leftrightarrow (4)$ is Theorem \ref{Theorem dichotomy}.
\end{proof}

\begin{rem}
\emph{If we were to consider K-polystability with respect to products in the sense that $\mathcal{X} \simeq X \times \mathbb{P}^1$ (for transcendental test configurations as in \cite{SD1,DervanRoss, Dervanrelative, SD2, ChuTosattiWeinkove}), then the corresponding K-polystable locus would not contain the cscK locus in general. In fact, whenever $\mathrm{Aut}_0(X) \neq \emptyset$ the K-polystable locus would always be empty, so $\mathcal{X}_{\pi^{-1}(\mathbb{C})} \simeq X \times \mathbb{C}$ yields the strongest notion of product configuration of this type that is worth considering.}
\end{rem}

\bigskip

\section{Weakly cscK manifolds and applications} \label{Section 5}

\subsection{The special case of weakly cscK manifolds}

In view of the above main results, it is interesting to study situations when some of the above mentioned stability loci are non-empty (and for which we will then be able to establish that certain stability notions must be equivalent). A natural candidate for such manifolds are those compact K\"ahler manifolds $(X,\omega)$ that admit a cscK metric \emph{in some possibly different K\"ahler class} $\alpha \neq [\omega] \in \mathcal{C}_X$. We will refer to such manifolds as weakly cscK. 

\begin{mydef}
We say that a compact K\"ahler manifold is \emph{weakly cscK} if the associated cscK locus $\neq \emptyset$.
\end{mydef}

\noindent Note that a manifold can be weakly cscK without being cscK. Examples of this phenomenon can in particular be obtained by any K\"ahler-Einstein manifolds which also admits K-unstable polarizations. Concretely, it was shown through a study of \emph{slope stability}, in \cite[Example 5.30]{RossThomas}, that e.g. $\mathbb{P}^2$ blown up in $8$ points in generic position satisfies this condition (see also \cite{Dervanalphainvariant} and \cite{CMG1, CMG2} for a more explicit treatment of this and other Del Pezzo surface examples).  
The idea is then to use the techniques of changing the underlying K\"ahler class, to reduce the study of arbitrary polarizations to the case when the underlying K\"ahler class admits a cscK metric.
Some noteworthy corollaries of Theorem \ref{Theorem dichotomy} and Theorem \ref{Theorem dichotomy main} follow:

\begin{thm} \label{Theorem polarized equivalence}
Let $(X,L)$ be a polarized weakly cscK manifold. Then the following holds:
\begin{enumerate}
    \item $(X,L)$ is  K-polystable if and only if it is  geodesically K-polystable.
    \item $(X,L)$ is equivariantly geodesically K-polystable if and only if it is equivariantly K-polystable
\end{enumerate}
\end{thm}

\noindent
For arbitrary compact K\"ahler manifolds $(X,\omega)$ the result $(1)$ is known to hold if the automorphism group is discrete, see \cite{SD2}, and the second point $(2)$ holds in general. In particular, we record the following result related to Remark \ref{Remark comparison stability}: 

\begin{thm}
Suppose that $X$ is a weakly cscK K\"ahler manifold with $\mathrm{Aut}_0(X)$ discrete. Then $(X,\alpha)$ is uniformly K-stable if and only if $(X,\alpha)$ is coercive with respect to the set of subgeodesic rays compatible with a relatively K\"ahler test configuration for $(X,\alpha)$. Likewise, $(X,\alpha)$ is K-stable if and only if $(X,\alpha)$ is geodesically stable with respect to the set of subgeodesic rays compatible with a relatively K\"ahler test configuration for $(X,\alpha)$.
\end{thm}

\begin{proof} [Proof of Theorem \ref{Theorem polarized equivalence}]
This is an immediate consequence of Theorem \ref{Theorem dichotomy main}. Indeed, under the stated hypotheses the K-polystable locus is non-empty, by results of \cite{Berman}, so $(X,L)$ is K-polystable if and only if it is geodesically K-polystable, even $c_1(L)$ does not itself admit a cscK metric. 
Finally, K-polystability trivially implies equivariant K-polystability, so also the equivariantly K-polystable locus is non-empty. In the same way as above, this proves $(2)$. 
\end{proof}

\noindent A reformulation of the above Theorem \ref{Theorem polarized equivalence} is that, on weakly cscK polarized manifolds, a ray compatible with a test configuration is induced by a holomorphic vector field precisely if the test configuration is a product: 

\begin{thm} \label{Theorem polarized equivalence cor} 
Suppose that $(X,L)$ is a polarized weakly cscK manifold. Let $(\mathcal{X},\mathcal{L})$ be a relatively K\"ahler test configuration for $(X,L)$ with compatible subgeodesic ray $(\varphi_t) \in \mathrm{R}(\mathcal{X},\mathcal{L})$. Then $\mathcal{X} \simeq X \times \mathbb{C}$ if and only if  $(\varphi_t)_{t \geq 0}$ is induced by a holomorphic vector field on $X$.
\end{thm}

\noindent This extends a result of \cite{Berman} from the case of cscK manifolds, to the larger class of weakly cscK manifolds. 

\subsection{An extended injectivity lemma}

Furthermore, it is worth noting that the above techniques can be used to extend the injectivity lemma (see Theorem \ref{Theorem summary}, part (2)) from the setting of a fixed underlying K\"ahler class, to the setting of different underlying K\"ahler classes $\alpha, \beta \in \mathcal{C}_X$. Such injectivity type results were in \cite{SD2} a key tool in proving equivariant K-polystability, geodesic K-polystability, and K-polystability whenever the automorphism group is discrete. It is also of independent interest. 

In order to state the result, recall the assignment $\mathrm{R}: (\mathcal{X}, \mathcal{A}) \mapsto \left[(\varphi_t)^{(\mathcal{X},\mathcal{A})}\right]$ from Section \ref{Section preliminary rays}. 
We then have the following: 

\begin{thm} \label{Theorem extended injectivity not intro}
Suppose that $\alpha := [\omega]$ and $\beta := [\theta]$ are K\"ahler classes on $X$ and let $(\mathcal{X},\mathcal{A})$ and $(\mathcal{Y}, \mathcal{B})$ be relatively K\"ahler test configurations for $(X,\alpha)$ and $(X,\beta)$ respectively. Suppose that $$\mathrm{R}(\mathcal{X},\mathcal{A}) \cap \mathrm{R}(\mathcal{Y}, \mathcal{B}) \neq \emptyset.$$ 
Then the canonical $\mathbb{C}^*$-equivariant isomorphism $\mathcal{X} \setminus \mathcal{X}_0 \rightarrow \mathcal{Y} \setminus \mathcal{Y}_0$  extends to an isomorphism $\mathcal{X} \rightarrow \mathcal{Y}$. 
\end{thm}

\begin{rem} \emph{The hypothesis $\mathrm{R}(\mathcal{X},\mathcal{A}) \cap \mathrm{R}(\mathcal{Y}, \mathcal{B}) \neq \emptyset$ here means that there is a subgeodesic ray $\rho(t) \in \mathrm{PSH}(X,\omega) \cap \mathrm{PSH}(X,\theta)$ which is compatible with two relatively K\"ahler test configurations $(\mathcal{X},\mathcal{A})$ and $(\mathcal{Y}, \mathcal{B})$ for $(X,\alpha)$ and $(X,\beta)$ respectively. } 
\end{rem}

\begin{proof}
[Proof of Theorem \ref{Theorem extended injectivity not intro}]
The idea of the proof is to extend \cite[Theorem 1.8]{SD2} using the key Lemma \ref{Lemma main} in order to control the change of the underlying K\"ahler class. Indeed, first fix K\"ahler forms $\omega_{\alpha}$ and $\omega_{\beta}$ such that $[\omega_{\alpha}] = \alpha$ and $[\omega_{\beta}] = \beta$. By hypothesis $\mathrm{R}(\mathcal{X},\mathcal{A}) \cap \mathrm{R}(\mathcal{Y}, \mathcal{B}) \neq \emptyset$ there is a subgeodesic ray $\rho(t) \in \mathrm{PSH}(X,\omega_{\alpha}) \cap \mathrm{PSH}(X,\omega_{\beta})$ which is compatible with two relatively K\"ahler test configurations $(\mathcal{X},\mathcal{A})$ and $(\mathcal{Y}, \mathcal{B})$ for $(X,\alpha)$ and $(X,\beta)$ respectively. 
Now pick $(\mathcal{X},\mathcal{A}_{\beta})$ as in Lemma \ref{Lemma main}. Then $(\mathcal{X},\mathcal{A}_{\beta})$ and $(\mathcal{Y},\mathcal{B})$ are relatively K\"ahler test configurations for $(X,\beta)$, both compatible with the same subgeodesic ray $\rho(t)$. By applying the injectivity lemma \cite[Theorem 1.8]{SD2} we then finally see that the canonical $\mathbb{C}^*$-equivariant isomorphism $\mathcal{X} \setminus \mathcal{X}_0 \rightarrow \mathcal{Y} \setminus \mathcal{Y}_0$  extends to an isomorphism $\mathcal{X} \rightarrow \mathcal{Y}$. This is what we wanted to prove.
\end{proof}

\subsection{Topology of the K-semistable and uniformly K-stable loci} \label{Section topology}

\noindent Finally, the techniques of variation of the underlying class in the K\"ahler cone immediately yield some basic information on the structure and topology of the K-semistable and uniformly K-stable loci in the K\"ahler cone. Here $(X,\alpha)$ is said to be uniformly K-stable if there is a $\delta > 0$ such that $\mathrm{DF}(\mathcal{X},\mathcal{A}) \geq \delta \mathrm{J}^{\mathrm{NA}}(\mathcal{X},\mathcal{A})$ for all relatively K\"ahler test configurations $(\mathcal{X}, \mathcal{A})$ for $(X,\alpha)$ that dominate $X \times \mathbb{P}^1$ via a morphism $\mu: \mathcal{X} \rightarrow X \times \mathbb{P}^1$ (testing for these is enough by \cite{Zakthesis}). For such test configurations the norm $\mathrm{J}^{\mathrm{NA}}(\mathcal{X},\mathcal{A})$ is defined as the intersection number
$$
\mathrm{J}^{\mathrm{NA}}(\mathcal{X},\mathcal{A}) := (\mu^*p_1^*\alpha \cdot \mathcal{A}) - \frac{(\mathcal{A}^{n+1})}{n+1}.
$$
computed on $\mathcal{X}$ (as before $p_1: X \times \mathbb{P}^1 \rightarrow X$ denotes the first projection). We refer to \cite{DervanRoss,SD1,SD2, Zakthesis} for details. 

The following first result should be compared to \cite[Theorem G]{Antoniso}: 

\begin{thm} \label{Theorem Kss closed}
The K-semistable locus is closed in Euclidean topology in the K\"ahler cone of $X$.
\end{thm}

\begin{proof}
By \cite[Proposition 3.12]{SD1} it suffices to test K-semistability for relatively K\"ahler test configurations $(\mathcal{X}, \mathcal{A})$ for $(X,\alpha)$ that are smooth and dominating, i.e. there is a morphism $\mu: \mathcal{X} \rightarrow X \times \mathbb{P}^1$ such that $p_1 \circ \mu = \pi: \mathcal{X} \rightarrow \mathbb{P}^1$.
Hence, we may fix any given relatively K\"ahler smooth and dominating test configuration $\mathcal{X}$ for $X$.
By \cite[Proposition 3.10]{SD1} we moreover have
$$
\mathcal{A} = \mu^*p_1^*\alpha + [D]
$$
for some $\mathbb{R}$-divisor $D$ on $\mathcal{X}$ supported on the central fiber $\mathcal{X}_0$. Since $\mathcal{A}$ is K\"ahler, there is an open neighbourhood $U_{\alpha} \subset \mathcal{C}_X$ of $\alpha$ such that $\mathcal{A}_{\beta} := \mu^*p_1^*\beta + [D] \in \mathcal{C}_X$ for every $\beta \in U_{\alpha}$.
In view of the intersection theoretic interpretation of the Donaldson-Futaki invariant, note that the map
$$
U_{\alpha} \ni \beta \mapsto \mathrm{DF}(\mathcal{X},\mathcal{A}_{\beta})
$$
is continuous. As a consequence, suppose that $\alpha \not \in$ K-semistable locus. Then there exists a smooth and dominating test configuration $(\mathcal{X}, \mathcal{A}_{\alpha})$ as above such that $\mathrm{DF}(\mathcal{X},\mathcal{A}_{\alpha}) < 0$. But by continuity there exists an open neighbourhood $V_{\alpha} \subset U_{\alpha}$ such that $\mathrm{DF}(\mathcal{X}, \mathcal{A}_{\beta}) < 0$ for each $\beta \in V_{\alpha}$. In other words, if $\alpha \not \in $ K-semistable locus, then there is an open neighbourhood satisfying $V_{\alpha} \subset \mathcal{C}_X \setminus \textrm{K-semistable locus}.$ \noindent Hence the K-semistable locus is open in the K\"ahler cone of $X$, which is what we wanted to prove. 
\end{proof}

\noindent Due to the fact that the cscK locus is open (see \cite{LeBrunSimanca}) a consequence of this is that K-semistability is not equivalent to existence of cscK metrics. This has been known previously by means of counterexamples (see e.g. \cite{Tian, Keller}). Nonetheless, this yields a complementary perspective on this question. From this, we we also record the following corollary of independent interest: 

\begin{cor} 
The inclusion $\textrm{cscK locus} \subset \textrm{K-semistable locus}$ is strict whenever the K-semistable locus $\neq \emptyset, \mathcal{C}_X$. 
\end{cor}

\noindent This also yields concrete examples of manifolds with ``many'' strictly semistable (i.e. K-semistable but not K-stable) K\"ahler classes:

\begin{ex} \emph{(Strictly semistable examples)}
\emph{Consider the Del Pezzo surface $X = \mathrm{Bl}_{p_1, \dots, p_8} \mathbb{P}^2$ to be the blowup of $\mathbb{P}^2$ in 8 points $p_1, \dots, p_8$ in general position. First of all, it is well known that $X$ is K\"ahler-Einstein, so $(X,-K_X)$ is K-stable by \cite{Tian}.
In other words, the K-stable locus, thus also the K-semistable locus, is non-empty. On the other hand, it was shown in \cite{RossThomas} that $X$ admits K-unstable polarizations, so the K-semistable locus is $\neq \mathcal{C}_X$. Since both the K-stable locus and the K-semistable locus are $\neq \emptyset, \mathcal{C}_X$, whereas the former is open and the latter is closed in the Euclidean topology in $\mathcal{C}_X$, it follows that the strictly K-semistable locus is non-empty, i.e. the set
$
\textrm{K-semistable locus} \setminus \textrm{K-stable locus} \neq \emptyset.
$ 
This gives a new method of answering the question of existence of strictly K-semistable classes. }
\end{ex}

\subsubsection{The uniformly K-stable locus}

Now suppose that $(X,\omega)$ is a compact K\"ahler manifold with discrete automorphism group, i.e $\mathrm{Aut}_0(X) = \emptyset$. Then similar arguments can also be made for the uniformly K-stable locus in the K\"ahler cone of $X$. To see this, we associate to each K\"ahler class $\alpha \in \mathcal{C}_X$ the finite real number
$$
\Delta(\alpha) := \sup \{ \delta > 0 \; \vert \mathrm{DF}(\mathcal{X},\mathcal{A}) \geq \delta \mathrm{J}^{\mathrm{NA}}(\mathcal{X},\mathcal{A})\}
$$
where the condition above should hold for all relatively K\"ahler test configurations $(\mathcal{X},\mathcal{A})$ for $(X,\alpha)$. Moreover, introduce the sets
$$
\mathcal{U}_{\delta} := \{ \alpha \in \mathcal{C}_X \; \vert \; \Delta(\alpha) \geq \delta \}
$$
\noindent We then make the following observation:

\begin{thm} \label{Theorem Uks closed}
The uniformly K-stable locus can be written as a union
$$
\mathcal{U} := \bigcup_{\delta > 0} \mathcal{U}_{\delta},
$$
where each $\mathcal{U}_{\delta}$ is closed in the Euclidean topology in the K\"ahler cone. 
\end{thm}

\begin{proof}
The proof is analogous to the one in Theorem \ref{Theorem Kss closed}, but applied to $\mathrm{DF} - \delta \mathrm{J}^{\mathrm{NA}}$ instead. For the convenience of the reader we give the argument: 
By \cite[Proposition 3.2.20]{Zakthesis} it suffices to test uniform K-stability for relatively K\"ahler test configurations $(\mathcal{X}, \mathcal{A})$ for $(X,\alpha)$ that are smooth and dominating, i.e. there is a morphism $\mu: \mathcal{X} \rightarrow X \times \mathbb{P}^1$ such that $p_1 \circ \mu = \pi: \mathcal{X} \rightarrow \mathbb{P}^1$.
Hence, we may fix any given relatively K\"ahler smooth and dominating test configuration $\mathcal{X}$ for $X$.
As before, by \cite[Proposition 3.10]{SD1} we moreover have
$
\mathcal{A} = \mu^*p_1^*\alpha + [D]
$
for some $\mathbb{R}$-divisor $D$ on $\mathcal{X}$ supported on the central fiber $\mathcal{X}_0$. Since $\mathcal{A}$ is K\"ahler, there is an open neighbourhood $U_{\alpha} \subset \mathcal{C}_X$ of $\alpha$ such that $\mathcal{A}_{\beta} := \mu^*p_1^*\beta + [D] \in \mathcal{C}_X$ for every $\beta \in U_{\alpha}$.
In view of the intersection theoretic interpretation of both the Donaldson-Futaki invariant and the non-Archimedean $\mathrm{J}$-functional, note that the map
$$
U_{\alpha} \ni \beta \mapsto \mathrm{DF}(\mathcal{X},\mathcal{A}_{\beta}) - \delta \mathrm{J}^{\mathrm{NA}}(\mathcal{X},\mathcal{A}_{\beta})
$$
is continuous for each $\delta \in \mathbb{R}$. As a consequence, fix a $\delta \in \mathbb{R}$ and suppose that $\alpha \not \in \mathcal{U}_{\delta}$. Then, by definition, there exists a smooth and dominating test configuration $(\mathcal{X}, \mathcal{A}_{\alpha})$ as above such that $\mathrm{DF}(\mathcal{X},\mathcal{A}_{\alpha}) < \delta \mathrm{J}^{\mathrm{NA}}(\mathcal{X},\mathcal{A}_{\beta}$. But by continuity there exists an open neighbourhood $V_{\alpha} \subset U_{\alpha}$ such that $\mathrm{DF}(\mathcal{X}, \mathcal{A}_{\beta}) < \delta \mathrm{J}^{\mathrm{NA}}(\mathcal{X},\mathcal{A}_{\beta}$ for each $\beta \in V_{\alpha}$. In other words, if $\alpha \not \in \mathcal{U}_{\delta}$, then there is an open neighbourhood satisfying $V_{\alpha} \subset \mathcal{C}_X \setminus \mathcal{U}_{\delta}.$ \noindent Hence for each $\delta \in \mathbb{R}$, the set $\mathcal{U}_{\delta}$ is closed in the K\"ahler cone of $X$. Finally, it is clear that the uniformly K-stable locus can be written
$$
\mathcal{U} := \bigcup_{\delta > 0} \mathcal{U}_{\delta},
$$
completing the proof. 
\end{proof}

\begin{rem}
The K-semistable locus equals 
$$
\mathrm{Kss} = \bigcup_{\delta \geq 0} \mathcal{U}_{\delta} \; (= \mathcal{U}_0).
$$
so Theorem \ref{Theorem Kss closed} is a special case of Theorem \ref{Theorem Uks closed}. 
\end{rem}

\noindent We finally note the following result of independent interest, which is an immediate consequence of the proof of Theorem \ref{Theorem Uks closed}:

\begin{cor}
The stability threshold 
$$
\mathcal{C}_X \ni \alpha \mapsto \Delta(\alpha)
$$
is upper semicontinous. 
\end{cor}

\noindent We expect that it is also lower semicontinuous (thus proving openness of uniform K-stability as we vary the K\"ahler class). We leave this interesting question for future work.


\bigskip 


\begin{thebibliography}{1}

\bibitem[1]{ArezzoTian} C. Arezzo and G. Tian, \emph{Infinite geodesic rays in the space of {K\"ahler} potentials}, Ann. Sc. Norm. Super. Pisa Cl. Sci. \textbf{2} (2003), no. 4, 617-630.

\bibitem[2]{BB} R. Berman and B. Berndtsson, \emph{Convexity of the {K}-energy on the space of {K\"ahler} metrics and uniqueness of extremal metrics}, J. Amer. Math. Soc. \textbf{30} (2017), 1165-1196.

\bibitem[3]{Berman} R. Berman, \emph{K-polystability of {Q}-{Fano} varieties admitting {K\"ahler}-{Einstein} metrics}, Invent. Math. \textbf{203} (2016), no. 3, 1-53.

\bibitem[4]{Bermantransc} R. Berman, \emph{From {Monge}-{Amp\`ere} equations to envelopes and geodesic rays in the zero temperature limit}, Preprint arXiv: 1307.3008 (2013).

\bibitem[5]{BBGZ} R. Berman and S. Boucksom and V. Guedj and A. Zeriahi, \emph{A variational approach to complex {Monge}-{Amp\`ere} equations}, Publ. Math. de l'IHES \textbf{117} (2013), 179-245.

\bibitem[6]{BDLconvexity} R. Berman and T. Darvas and C. Lu, \emph{Convexity of the extended {K}-energy and the large time behaviour of the weak {Calabi} flow}, Geom. Topol. \textbf{21} (2017), no. 5, 2945-2988. 

\bibitem[7]{BDL} R. Berman and T. Darvas and C. Lu, \emph{Regularity of weak minimizers of the {K-energy} and applications to properness and {K}-stability}, Preprint arXiv:1602.03114v1, (2016).

\bibitem[8]{Boucksomsurvey} S. Boucksom, \emph{Variational and non-archimedean aspects of the Yau-Tian-Donaldson conjecture}, Preprint arXiv:1805.03289 (2018). 

\bibitem[9]{BHJ1} S. Boucksom and T. Hisamoto and M. Jonsson, \emph{Uniform {K}-stability, {Duistermaat}-{Heckman} measures and singularities of pairs}, Ann. Inst. Fourier (Grenoble) \textbf{67} (2015), no. 2. 

\bibitem[10]{BHJ2} S. Boucksom and T. Hisamoto and M. Jonsson, \emph{Uniform {K}-stability and asymptotics of energy functionals in {K\"ahler} geometry}, Preprint arXiv:1603.01026, (2016).

\bibitem[11]{LeBrunSimanca} C. Le Brun and S.R. Simanca, \emph{Extremal {K\"ahler} metrics and complex deformation theory}, Geom. Funct. Anal. \textbf{4} (1994), 298-336.

\bibitem[12]{Calabiextremal} E. Calabi, \emph{Extremal {K\"ahler} metrics} Seminar on Differential Geometry, Ann. of Math. Stud., 102, Princeton Univ. Press (1982), 259-290.

\bibitem[13]{lecturessymplectic} A. Cannas da Silva, \emph{Lectures on {Symplectic} {Geometry}}, Lecture notes in mathematics \textbf{1764}, Springer-Verlag Berlin Heidelberg (2008).

\bibitem[14]{CMG1} I. Cheltsov and J. Martinez-Garcia, \emph{Stable polarized del Pezzo surfaces}, Preprint arXiv:1606.04370.

\bibitem[15]{CMG2} I. Cheltsov and J. Martinez-Garcia, \emph{Unstable polarized del Pezzo surfaces. }, Preprint arXiv:1707.06177.

\bibitem[16]{ChenChengI} X.X. Chen and J. Cheng, \emph{On the constant scalar curvature {K\"ahler} metrics, apriori estimates}, arXiv:1712.06697 (2017).

\bibitem[17]{ChenChengII} X.X. Chen and J. Cheng, \emph{On the constant scalar curvature {K\"ahler} metrics, existence results}, arXiv:1801.00656 (2018).

\bibitem[18]{ChenChengIII} X.X. Chen and J. Cheng, \emph{On the constant scalar curvature {K\"ahler} metrics, general automorphism group}, arXiv:1801.05907 (2018).

\bibitem[19]{CDSone} X.X. Chen, S. Donaldson and S. Sun, \emph{{K\"ahler}-{Einstein} metrics on {Fano} manifolds {I}: approximation of metrics with cone singularities}, J. Amer. Math. Soc. \textbf{28} (2015), 183-197.

\bibitem[20]{CDStwo} X.X. Chen, S. Donaldson and S. Sun, \emph{{K\"ahler}-{Einstein} metrics on {Fano} manifolds {II}: limits with cone angle less than 2pi}, J. Amer. Math. Soc. \textbf{28} (2015), 199-234.

\bibitem[21]{CDSthree} X.X. Chen, S. Donaldson and S. Sun, \emph{{K\"ahler}-{Einstein} metrics on {Fano} manifolds {III}: limits as cone angle approaches 2pi and completion of the main proof}, J. Amer. Math. Soc. \textbf{28} (2015), 235-278.

\bibitem[22]{ChenLiPaun} X.X. Chen, L. Li and M. Paun, \emph{Approximation of weak geodesics and subharmonicity of {Mabuchi} energy}, Ann. Fac. Sci. Toulouse. Math. \textbf{25} (2016), no. 5, 935-957.

\bibitem[23]{ChenTang} X.X. Chen and Y. Tang, \emph{Test configurations and {Geodesic} rays}, Preprint arXiv:0707.4149 (2007).

\bibitem[24]{ChuTosattiWeinkove} J. Chu, V. Tosatti and B. Weinkove, \emph{On the {$C^{1,1}$} regularity of geodesics in the space of {K\"ahler} metrics}, Comm. Partial Differential Equations 10.1080/03605302.2018.1446167 (2017).

\bibitem[25]{Darvas14} T. Darvas, \emph{The {Mabuchi} completion of the space of {K\"ahler} potentials}, Amer. J. Math, 10.1353/ajm.2017.0032 (2014).

\bibitem[26]{Darvas15} T. Darvas, \emph{The {Mabuchi} Geometry of finite energy classes}, Adv. Math. \textbf{285} (2015), 182-219.

\bibitem[27]{Darvassurvey} T. Darvas, \emph{Geometric pluripotential theory on {K\"ahler} manifolds}, Survey article (2017).

\bibitem[28]{Darvas} T. Darvas, \emph{Weak geodesic rays in the space of {K\"ahler} potentials and the class {$E(X, \omega_0)$}}, J. Inst. Math. Jussieu \textbf{16} (2017), no. 4, 837-858.

\bibitem[29]{DR} T. Darvas and Y.A. Rubinstein, \emph{Tian's properness conjecture and {Finsler} geometry of the space of {K\"ahler} metrics}, J. Amer. Math. Soc. \textbf{30} (2017), no. 2, 347-387.

\bibitem[30]{GaborDatar} V. Datar and G. Sz\'ekelyhidi, \emph{{K\"ahler}-{Einstein} metrics along the smooth continuity method}, Geom. Funct. Analysis \textbf{26} (2016), no. 4, 975-1010.

\bibitem[31]{Dervanalphainvariant} R. Dervan, \emph{Alpha invariants and {K}-stability for general polarisations of {Fano}
varieties}, Int. Math. Res. Not. (2015), no. 16, pp. 7162-7189.

\bibitem[32]{Dervanrelative} R. Dervan, \emph{Relative {K}-stability for {K\"ahler} manifolds}, Math. Ann. https://doi.org/ 10.1007/s00208-017-1592-5 (2017).

\bibitem[33]{DervanRoss} R. Dervan and J. Ross, \emph{K-stability for {K\"ahler} manifolds}, Math. Res. Lett. \textbf{24} (2017).

\bibitem[34]{Donaldsongeodesicequation} S.K. Donaldson, \emph{Symmetric spaces, {K\"ahler} geometry and {Hamiltonian} dynamics}, Northern California Symplectic Geometry Seminar \textbf{2} (1999), 13-33.

\bibitem[35]{Donaldsontoric} S.K. Donaldson, \emph{Scalar curvature and stability of toric varieties}, J. Diff. Geom. \textbf{62} (2002), 289-349.

\bibitem[36]{Antoniso} A. Isopoussu, \emph{K-stability of relative flag varieties}, PhD thesis, arXiv:1307.7638 (2013).

\bibitem[37]{Keller} J. Keller, \emph{About canonical {K\"ahler} metrics on Mumford semistable projective bundles over a curve}, J. London Math. Soc. \textbf{93} (2016), no. 1, pp. 159-174.

\bibitem[38]{MabuchiKenergy1} T. Mabuchi, \emph{A functional integrating {Futaki} invariants}, Proc. Japan Acad. \textbf{61} (1985), 119-120.

\bibitem[39]{Mabuchi} T. Mabuchi, \emph{K-energy maps integrating {Futaki} invariants}, Tohoku Math. J. \textbf{38} (1986), no. 4, 575-593.

\bibitem[40]{Mabuchisymplectic} T. Mabuchi, \emph{Some symplectic geometry on compact {K\"ahler} manifolds {I}}, Osaka J. Math. \textbf{24} (1987), 227-252.

\bibitem[41]{Odaka} Y. Odaka, \emph{A generalization of the {Ross} {Thomas} slope theory}, Osaka J. Math. \textbf{50} (2013), no. 1, 171-185.

\bibitem[42]{PRS} D. H. Phong, J. Ross and J. Sturm, \emph{Deligne pairings and the {Knudsen}-{Mumford} expansion}, J. Diff. Geom. \textbf{78} (2008), no. 3, 475-496.

\bibitem[43]{RossThomas} J. Ross and R.P. Thomas, \emph{An obstruction to the existence of constant scalar curvature {K\"ahler} metrics}, J. Diff. Geom. \textbf{72} (2006), 429-466.

\bibitem[44]{WR} J. Ross and D. Witt Nystr\"om, \emph{Analytic test configurations and geodesic rays}, J. Symplectic Geom. \textbf{12} (2014), 125-169.

\bibitem[45]{SD1} Z. {Sj\"ostr\"om} Dyrefelt, \emph{K-semistability of csc{K} manifolds with transcendental cohomology class}, J. Geom. Anal. https://doi.org/10.1007/s12220-017-9942-9 (2017).

\bibitem[46]{SD2} Z. {Sj\"ostr\"om} Dyrefelt, \emph{On K-polystability of csc{K} manifolds with transcendental cohomology class}, Int. Math. Res. Not. doi:10.1093/imrn/rny094 (2018).

\bibitem[47]{Zakthesis} Z. {Sj\"ostr\"om} Dyrefelt, \emph{{K-stabilit\'e} et {vari\'et\'es} {k\"ahleriennes} avec classe transcendante}, PhD thesis http://thesesups.ups-tlse.fr/3577/ (2017).

\bibitem[48]{GaborStoppa} J. Stoppa and G. Sz{\'e}kelyhidi, \emph{Relative K-stability of extremal metrics}, J. Eur. Math. Soc. \textbf{13} (2011) no. 4, p. 899--909

\bibitem[49]{Gabornotes} G. Sz{\'e}kelyhidi, \emph{Introduction to extremal {K\"ahler} metrics}, Graduate studies in mathematics vol. 152, American Mathematical Society, Providence RI (2014). 

\bibitem[50]{Tian} G. Tian, \emph{{K\"ahler}-{Einstein} metrics with positive scalar curvature}, Invent. Math. \textbf{130} (1997), no. 1, 1-37.

\bibitem[51]{Tianlecturenotes} G. Tian, \emph{Canonical metrics in {K\"ahler} geometry}, Lectures in Mathematics ETH Zurich, Birkhauser Verlag, Basel (2000).

\bibitem[52]{TianYTDconjecture} G. Tian, \emph{K-stability and {K\"ahler}-{Einstein} metrics.}, Communications on Pure and Applied Math. \textbf{68} (2015), 1085-1156.

\bibitem[53]{Wang} X. Wang, \emph{Height and {GIT} weight}, Math. Res. Lett. \textbf{19} (2012), no. 4, 909-926.

\end{thebibliography}
\end{document}